\documentclass{amsart}
\usepackage[utf8]{inputenc}
\usepackage{amsmath, amssymb, amsthm, amsfonts, mathtools, array, xfrac, float, multirow, indentfirst, url, enumerate, comment, afterpage, pifont, makecell, xcolor, hyperref}
\usepackage[backend=bibtex, sorting=none, bibstyle=alphabetic, citestyle=alphabetic, sorting=nyt,maxbibnames=99, giveninits=true, isbn=false, url=false, doi=false, eprint=false]{biblatex}

\bibliography{references.bib}

\newtheorem{theorem}{Theorem}
\newtheorem{lemma}[theorem]{Lemma}
\newtheorem{proposition}[theorem]{Proposition}

\theoremstyle{definition}
\newtheorem*{remark}{Remark}

\renewcommand{\Re}{\textrm{Re}}

\begin{document}

\title{Explicit bounds for the Riemann zeta-function on the 1-line}
\author[G.\ Hiary, N.\ Leong, \and A.\ Yang]{Ghaith A. Hiary, Nicol Leong, and Andrew Yang}

\address{
    GH: Department of Mathematics, The Ohio State University, 231 West 18th
    Ave, Columbus, OH 43210, USA
}
\email{hiary.1@osu.edu}
\address{
    NL: University of New South Wales (Canberra) at the Australian Defence Force Academy, ACT, Australia
}
\email{nicol.leong@unsw.edu.au}
\address{
    AY: University of New South Wales (Canberra) at the Australian Defence Force Academy, ACT, Australia
}
\email{andrew.yang1@unsw.edu.au}

\subjclass[2020]{Primary: 11M06, 11Y35, 11L07.}
\keywords{Van der Corput estimate, exponential sums, Riemann zeta function.}

\begin{abstract}
Explicit estimates for the Riemann zeta-function on the $1$-line are derived
using various methods, in particular van der Corput lemmas of high order and a theorem of Borel and Carath\'{e}odory.
\end{abstract}

\maketitle

\section{Introduction}

Let $\zeta(s)$ denote the Riemann zeta-function. An open problem in analytic number theory is to determine the growth rate of $\zeta(s)$ on the line $s = 1 + it$ as $t \to +\infty$. Assuming the Riemann hypothesis, Littlewood \cite{littlewood_on_1926, littlewood_on_1928} showed that the value of $|\zeta(1+it)|$ is limited to the range
$$\frac{1}{\log\log t}\ll \zeta(1 + it) \ll \log \log t.$$ 
In comparison, the sharpest estimates requiring no assumption furnish a wider range of the form 
\begin{equation}\label{current form}
    \frac{1}{(\log t)^{2/3}(\log\log t)^{1/3}} \ll \zeta(1 + it) \ll \log^{2/3}t.
\end{equation}

The upper bound in the unconditional estimate \eqref{current form} was made explicit by Ford \cite{ford_vinogradov_2002}. He computed numerical constants $c$ and $t_0$ such that\footnote{Trudgian \cite{trudgian_new_2014} later lowered Ford's value of $c$ from $76.2$ to $62.6$.}
$$|\zeta(1 + it)| \le c \log^{2/3}t,\qquad (t\ge t_0).$$
The constants $c$ and $t_0$ are
 often too large for standard applications, 
 including explicit bounds on $S(t)$ \cite{trudgian_improved_2014, hasanalizade2022counting} and explicit zero-density estimates \cite{kadiri_explicit_2018}. 
 One instead  defaults to 
an asymptotically worse bound of the form $\zeta(1 + it) \ll \log t$  
that nevertheless is sharper in a finite range of $t$ of interest. 

This motivates our Theorem~\ref{main theorem}. 
We establish a new explicit bound on $\zeta(1 + it)$ that is 
asymptotically sharper than $O(\log t)$ but still good at small values of $t$.  
\begin{theorem}\label{main theorem}
For $t \ge 3$, we have 
$$|\zeta(1 + it)| \le 1.731\frac{\log t}{\log \log t}.$$
\end{theorem}

Additionally, using a different method,
we prove explicit upper bounds of the same form, 
on both $1/|\zeta(1+it)|$ and
$\left|\zeta'(1+it)/\zeta(1+it)\right|$.

\begin{theorem}\label{main theorem 1}
For $t \ge 3$, we have 
\begin{equation*}
\frac{1}{|\zeta(1 + it)|} \le 431.7 \frac{\log t}{\log \log t}.
\end{equation*}
Furthermore, for $t\ge 500$, if
\begin{equation*}
1 \le \sigma \le 1+ \frac{9}{250}\frac{\log\log t}{\log t},
\end{equation*}
then
\begin{equation*}
\left|\frac{\zeta'}{\zeta}(\sigma+it)\right| \le 154.5 \frac{\log t}{\log \log t}.
\end{equation*}
\end{theorem}

Theorem~\ref{main theorem} makes explicit a result due to Weyl \cite{weyl_zur_1921}
 (see also \cite[Section 5.16]{titchmarsh_theory_1986}),  
and is sharper than the prior explicit bounds due to Patel \cite{patel_explicit_2022} on $|\zeta(1 + it)|$ 
for $t$ in the range
$\exp(3382) \le t \le \exp(3.61\cdot 10^8)$.
\footnote{We also note that Theorem~\ref{main theorem} is sharper 
for $t\ge 23256$ than the bound $|\zeta(1+it)|\le \tfrac{3}{4}\log t$ given in \cite{trudgian_new_2014}.} 

As far as we know, the bounds in Theorem~\ref{main theorem 1} are the first unconditional explicit bounds of their kind. The first assertion of Theorem \ref{main theorem 1} supersedes \cite[Proposition A.2]{carneiro2022optimality} when $t\ge \exp(\exp(10.07))$, while the second assertion supersedes \cite[Table 2 with $(W,R_1) = (12,40.14)$]{trudgian2015explicit} for $t\ge \exp(\exp(3.85))$.

Under the assumption of the Riemann Hypothesis, sharper explicit estimates for $\zeta'/\zeta(1 + it)$ and $1/\zeta(1+it)$ of the order $\log\log t$ are known \cite{simonic_estimates_2022}, \cite[Theorem\ 5]{chirre_conditional_2022}, \cite{chirre_conditional_2023}. Much work has also been done on bounding more general $L$-functions \cite{lumley_explicit_2018}, \cite{palojarvi_conditional_2022}.

Lastly, it is reasonable to try to adapt the remarkable approaches in 
\cite[p.\ 984]{sound_2009} and \cite[Lemma 2.5]{lxs_2015} 
to our context. In these articles,
very precise bounds, conditional on the Riemann hypothesis, are obtained.
One can arrange for these approaches 
to yield unconditional bounds,
and the outcome is similarly constrained by the width of 
available explicit zeros-free regions. 
See \cite{li_2010}, for example, where the method from \cite{sound_2009} is used to establish bounds on $L$-functions for $\sigma \ge 1$.

\section{Discussion}
One can derive an explicit bound on $|\zeta(1+it)|$  
of the form  $\log t + \alpha_1$
using a version of the Euler--Maclaurin summation.
For example, using the version in \cite[Corollary\ 2]{simonic_explicit_2020},  
applied with $t_0=30\pi$ and $c = 1/4$, together with the elementary harmonic sum bound\footnote{Here, $\gamma=0.577216\ldots$ is the Euler constant.} (see Lemma \ref{harmonic sum} below),
\begin{equation}\label{harmonic bound}
    \sum_{n = 1}^N\frac{1}{n}\le \log N+\gamma+\frac{1}{2N},
\end{equation} 
yields after a small numerical computation, 
$|\zeta(1+it)|\le \log t -0.45$ for $t\ge 3$.

Patel~\cite{patel_explicit_2022} 
replaced the use of the Euler--Maclaurin summation
by an explicit version of the Riemann--Siegel formula on the $1$-line.  
Consequently, the size of the main term in the Euler--Maclaurin bound was cut
by half. Patel thus obtained
$|\zeta(1+it)|\le \frac{1}{2}\log t +1.93$ for $t\ge 3$,  
a substantial improvement at the expense of a larger constant term.

In general, we expect that an explicit van der Corput lemma of order $k$, 
where $k\ge 2$, 
will produce a bound of the form
\begin{equation}\label{k bound}
|\zeta(1+it)|\le \frac{1}{k}\log t +\alpha_k,\qquad(t\ge 3).
\end{equation}
One proceeds by approximating $\zeta(1+it)$ by an Euler--Maclaurin sum
$\sum_{n\ll t} n^{-1-it}$, or better yet by a Riemann--Siegel sum $\sum_{n\ll \sqrt{t}} n^{-1-it}$.
In either case, roughly speaking, the term $(1/k)\log t$ arises from bounding the subsum 
 with $n\le t^{1/k}$ via the triangle inequality and
 the harmonic sum estimate \eqref{harmonic bound}, 
 while the constant
$\alpha_k$ predominantly arises from bounding the subsum with $n > t^{1/k}$ via 
van der Corput lemmas of order $\le k$. 
 
However, the $\alpha_k$ obtained using this approach appear to grow fast with $k$.
For example, in \cite{patel_explicit_2021}, $\alpha_5$ is about $23$-times larger than $\alpha_2$, 
which already suggests exponential growth.
Despite this, there is still gain in letting $k$ grow with $t$, provided the growth
is slow enough.
We let $k$ grow like $\log \log t$, which is permitted 
 by a recent optimized van der Corput test, uniform in $k$, derived in \cite{yang_explicit_2023}. 
Subsequently, we obtain
a bound on $|\zeta(1+it)|$ of the form $c_0 \log t/\log \log t$ with an explicit number $c_0$. 

While we could let $k$ increase faster than $\log \log t$, leading to a smaller 
contribution of the $(1/k)\log t$ term, 
the $\alpha_k$ thus obtained would likely take over as the main term, 
asymptotically surpassing $\log t/\log\log t$ in size. This limits 
the order of the van der Corput lemmas we employ.

In proving Theorem~\ref{main theorem}, we encounter an unexpected
``gap interval'', not covered by any of the other 
explicit estimates, which requires a special treatment.
Namely, the interval $I=[\exp(16), \exp(662)]$. 
To prove the inequality in Theorem~\ref{main theorem} for all $t\in I$, 
we use the following proposition.

\begin{proposition}\label{third_deriv_bound}
We have 
\[
|\zeta(1 + it)| \le \begin{cases}
\displaystyle    \frac{1}{3}\log t + 4.66,& t \ge \exp(16),\\
&\\
\displaystyle  \frac{8}{33}\log t + 12.45,& t \ge \exp(88).
\end{cases}
\]
\end{proposition}

To clarify the need for Proposition~\ref{third_deriv_bound}, 
consider that 
if one inputs the bounds in \cite{patel_explicit_2022}, as below, one finds
\[
\max_{t\in I} \frac{\displaystyle\min\left(\log t, \frac{1}{2}\log t + 1.93, \frac{1}{5}\log t + 44.02\right)}{\log t / \log \log t} \ge 2.539,
\]
and the maximum of this ratio occurs around $t_1 = \exp(140.3)$. 
On the other hand, the savings we otherwise  establish here 
may materialize only when $t$ is well past $t_1$.
Therefore,  without the improvement enabled by Proposition~\ref{third_deriv_bound}, 
we cannot do any better than the bound
$|\zeta(1 + it)| \le 2.539 \log t/\log \log t$ over $t\in I$. 
Proposition~\ref{third_deriv_bound}, which is based on a hybrid of van 
der Corput lemmas of orders $k=3$, $4$ and $5$, allows us to circumvent this barrier.

In retrospect, a gap interval like $I$ should be expected. This is because  
the estimates in \cite{patel_explicit_2021} rely, essentially, on van der Corput lemmas 
of orders $k=2$ and $k=5$.\footnote{The bound obtained via 
the Riemann--Siegel formula in \cite{patel_explicit_2021} 
should correspond to a $k=2$ van der Corput lemma.} In particular, 
bounds that would follow from using van der Corput lemmas 
of orders $k=3$ and $k=4$ 
are not utilized, which presumably causes the gap.

\section{Proof of Theorem~\ref{main theorem}}\label{Theorem 1 proof}

Let $t_0$ be a positive number to be chosen later, subject to $t_0\ge \exp(990/7)$. 
We suppose $t \ge t_0$ throughout this section. 
(We will use another method for $t < t_0$.) 

Theorem~\ref{patel theorem} supplies 
a Riemann--Siegel approximation of $\zeta(1+it)$ in terms of two sums, plus a remainder term
$\mathcal{R}$.
We use the triangle inequality to bound the entire second sum in this theorem. This gives
\begin{equation}\label{zeta bound 1}
    \begin{split}
        |\zeta(1+it)|&\le \left|\sum_{n=1}^{n_1} \frac{1}{n^{1+it}}\right|
+\frac{g(t)}{\sqrt{2\pi}}+\mathcal{R}(t),\qquad n_1=\lfloor \sqrt{t/(2\pi)}\rfloor.
    \end{split}
\end{equation}
Since $g(t)$ and $\mathcal{R}(t)$ are well-understood, 
the bulk of the work is in bounding the main sum over $n\le n_1$.

We introduce two sequences, $\{X_m\}_{0\le m\le M}$ and $\{Y_k\}_{3\le k\le r+1}$. 
The former sequence $X_m$ will be used to partition the main sum 
in \eqref{zeta bound 1} in a dyadic manner. 
The latter sequence $Y_k$ will be used to easily bound the partition points $X_m$ 
 from above and below  
solely in terms of $t$ (or, what is essentially the same, in terms of $n_1$).

To this end, let $h$ and $h_2$ be real positive numbers, 
also to be chosen later, subject to $1 < h \le 2$ and $h_2 < 1/\log 2$.
We use $h_2$ immediately in defining
\begin{equation*}
    \begin{split}
        &R := 2^{r - 1},\qquad 
        r := \lfloor h_2\log\log t \rfloor.
    \end{split}
\end{equation*}
Furthermore we will assume that $h_2$ and $t_0$ are chosen so that 
\begin{equation}\label{r0_assumption}
r_0 := \lfloor h_2 \log\log t_0 \rfloor \ge 6. 
\end{equation}
Since $t\ge t_0$ by supposition, $r\ge r_0$. The integer $r$ will correspond to the highest order van der Corput lemma 
we use in bounding the tail of the main sum in \eqref{zeta bound 1}. 

Furthermore, let
\begin{equation*}
\begin{split}
    &K := 2^{k - 1},\qquad \theta_k := \frac{2K}{(k-2)K + 4},\qquad (k \ge 3).
\end{split}
\end{equation*}
In the definition of $\theta_k$, $K$ is determined by $k$. 
Also, $\theta_k$ 
decreases monotonically to zero with $k\ge 3$, 
and is bounded from above by $\theta_3=1$.

We now define the previously mentioned sequences $\{X_m\}$ and $\{Y_k\}$. Set
\begin{equation*}
    \begin{split}
        &X_0 := \frac{1}{\sqrt{2\pi}}t^{1/2},\qquad X_m := h^{-m}X_0,\qquad (m \ge 1).
    \end{split}
\end{equation*}
So, $\lfloor X_0\rfloor = n_1$ and since $h>1$, $X_m$ decreases monotonically with $m$. 
Also set
\begin{equation}\label{Yk_defn}
    Y_k := X_0^{\theta_k},\qquad (k \ge 3),
\end{equation}
so that $Y_k$ decreases monotonically with $k\ge 3$, 
starting at $Y_3=X_0$. Note that
if $k$ is large, then $\theta_k\approx 2/k$, 
and so $Y_k \approx t^{1/k}$ for large $k$.
Lastly, define
\begin{equation*}
    M := \left\lceil \frac{(1 - \theta_{r + 1})\log X_0}{\log h} \right\rceil.
\end{equation*}
The integer $M$ corresponds to the number of points we will use to partition the main sum. From the definition, we have
\begin{equation}\label{M inequality}
  \frac{(1 - \theta_{r + 1})\log X_0}{\log h} \le  M < \frac{(1 - \theta_{r + 1})\log X_0}{\log h}+1.
\end{equation}
Thus, our choice of $M$ ensures that
\begin{equation}\label{XM_bound}
X_{M} \le X_0^{\theta_{r+1}}=Y_{r+1} < X_{M - 1} \le h X_0^{\theta_{r+1}}.
\end{equation}
Additionally, we have
\begin{equation}\label{XM_bound 2}
    X_M \le Y_{r + 1} < Y_r < \cdots < Y_4 < Y_3 = X_0.
\end{equation}

The assumption \eqref{r0_assumption} is required for the subsequent analysis to be valid. 
This condition implies, among other things,  
that $\theta_{r+1}<1$ and, hence, $M\ge 1$. 
This in turn ensures that some quantities appearing later 
(e.g.\ $M-1$) are nonnegative. 
%
All contingent conditions we impose 
 (including on $r_0$, $h$, $h_2$ and $t_0$) will be verified at the end, after choosing
the values of our free parameters.

By the chains of inequalities \eqref{XM_bound} and \eqref{XM_bound 2}, 
for all $m\in [1, M-1]$, $Y_{r+1} < X_m < Y_3$. 
Therefore, since the sequence $Y_{r+1},\ldots,Y_3$ partitions the interval 
$[Y_{r+1},Y_3]$, for each 
integer $m \in [1, M - 1]$, there is a unique integer $k \in[3, r]$ such that
\begin{equation}\label{mk dependence}
    X_m \in (Y_{k + 1}, Y_k].
\end{equation}
Note, though, that multiple $m$'s could correspond to the same $k$.

With this in mind, let us denote
\begin{equation}\label{Sm def}
S_m := \sum_{X_m + 1 < n \le X_{m - 1}}\frac{1}{n^{1 +it}} = \sum_{\lfloor X_m\rfloor + 1 < n \le \lfloor X_{m - 1}\rfloor}\frac{1}{n^{1 +it}},
\end{equation}
and observe that
\begin{equation}\label{Sm sum bound}
\left|\sum_{X_{M - 1} < n \le X_0}\frac{1}{n^{1 + it}}\right| \le \sum_{m = 1}^{M - 1}\left|\frac{1}{(\lfloor X_m\rfloor + 1)^{1 + it}} + S_m\right| \le \sum_{m = 1}^{M - 1}\left(\frac{1}{X_m} + |S_m|\right).
\end{equation}
We calculate
\begin{equation}\label{XM_XM1_ratio0}
\frac{\lfloor X_{m - 1}\rfloor}{\lfloor X_{m}\rfloor + 1} 
< \frac{X_{m - 1}}{X_m} = h.
\end{equation}
Here, we used the elementary inequality $x-1< \lfloor x \rfloor \le x$, valid for any real number $x$
together with the definition $X_m = h^{-m}X_0$.

We bound $S_m$ using Lemma~\ref{CD_deriv_test} with $a = \lfloor X_m\rfloor + 1$ and $b = \lfloor X_{m - 1} \rfloor \le ha$.
For each integer $m \in [1, M - 1]$ such that $X_m \in (Y_{k + 1}, Y_k]$, 
and with $K=2^{k-1}$, we have, for any choice of the free parameter $\eta > 0$,  
\begin{equation}\label{CkDk_bound_11}
\begin{split}
\left|S_m\right| 
&\le C_k(\eta, h) (\lfloor X_{m} \rfloor + 1)^{-k/(2K - 2)}t^{1/(2K - 2)} \\
&\qquad + D_k(\eta, h)(\lfloor X_m\rfloor + 1)^{k/(2K - 2)- 2/K}t^{-1/(2K - 2)}.
\end{split}
\end{equation}

To bound the RHS, observe that 
\begin{equation}\label{second term X_m bound}
Y_{k + 1} < \lfloor X_m\rfloor + 1 \le X_m + 1 = X_m \left(1 + \frac{1}{X_m}\right) \le Y_k\left(1 + \frac{1}{X_m}\right).
\end{equation}
In addition, $X_m \ge X_{M - 1} > X_0^{\theta_{r + 1}}$. Moreover, it follows by definition that 
\begin{equation}\label{theta lb 1}
    \theta_{r+1} = \frac{2}{r-1+2^{2-r}}.
\end{equation}
Since $\theta_{r+1}$ decreases continuously with $r$, 
and $r\le h_2\log\log t$, we have $X_0^{\theta_{r + 1}} = \exp(\theta_{r + 1}\log X_0) \ge \kappa(h_2, t)$ where 
\begin{equation}\label{const term bound 457}
\kappa(h_2, t) := \exp\left(\frac{\log (t/(2\pi))}{h_2\log \log t -1+2^{2-h_2\log\log t}}\right).
\end{equation} 
It is straightforward to verify that $\kappa(h_2, t)$ is increasing in $t$ for $t \ge t_0$, so $\kappa(h_2, t) \ge \kappa(h_2, t_0)$. Inserting \eqref{second term X_m bound} and \eqref{const term bound 457} into \eqref{CkDk_bound_11},
\begin{equation}\label{CkDk_bound_1}
\begin{split}
\left|S_m\right| &\le C_k(\eta, h)Y_{k + 1}^{-k/(2K - 2)}t^{1/(2K - 2)} \\
&\qquad + \widehat{D}_k(\eta, h, h_2, t_0) Y_{k}^{k/(2K - 2)- 2/K}t^{-1/(2K - 2)}
\end{split}
\end{equation}
for $t \ge t_0$, where 
\begin{equation}\label{D hat defn}
\widehat{D}_k(\eta, h, h_2, t_0) := D_k(\eta, h)\left(1 + \frac{1}{\kappa(h_2, t_0)}\right)^{k/(2K - 2)- 2/K}.
\end{equation}
This is permissible since the exponents of $\lfloor X_m\rfloor + 1$ in the first and second terms of \eqref{CkDk_bound_11} respectively satisfy
\begin{equation}\label{exponent bound}
-\frac{k}{2K - 2} < 0,\qquad\text{and}\qquad \frac{k}{2K - 2} - \frac{2}{K} \ge 0,\qquad (k \ge 3).
\end{equation}
It is worth clarifying that the right-hand side of \eqref{CkDk_bound_1} depends on $m$
via the requirement \eqref{mk dependence}.

Next, we express the bound on $S_m$ in \eqref{CkDk_bound_1} solely in terms of $t$.
Starting with the first term, we have
\begin{align}\label{Yk_bound}
Y_{k + 1}^{-k/(2K - 2)}t^{1/(2K - 2)} = (2\pi)^{\theta_{k+1}k/(4K - 4)}t^{-\rho(k)},
\end{align}
where after a quick calculation, 
\begin{equation}\label{rho def}
    \rho(k):=\frac{K-2}{K-1}\cdot \frac{1}{2(k-1)K+4}.
\end{equation}
We estimate the second term in \eqref{CkDk_bound_1} similarly, yielding
\begin{equation}\label{Yk_bound1}
\begin{split}
    Y_{k}^{k/(2K - 2)- 2/K}t^{-1/(2K - 2)} &< (2\pi)^{-\theta_k (k/(4K - 4)-1/K)} t^{-\frac{K}{K-1}\cdot\frac{1}{2(k-2)K+4}} \\
    & < (2\pi)^{-\theta_k (k/(4K - 4)-1/K)} t^{-\rho(k)}.
\end{split}
\end{equation}
The exponent of $2\pi$ in the estimate \eqref{Yk_bound}
may be bounded for $k> 5$ via $\theta_{k+1}\le 32/81$.
Additionally, recalling
the second inequality of \eqref{exponent bound}, we may
bound the contribution of the $2\pi$-factor
in the estimate \eqref{Yk_bound1}
simply by $1$. This motivates us to define
\begin{equation}\label{A def}
     A(k) := 
     \begin{cases}
     (2\pi)^{\theta_{k+1}\frac{k}{4(K - 1)}}C_k+ (2\pi)^{-\theta_k (\frac{k}{4(K - 1)}-\frac{1}{K})} \widehat{D}_k, & 3\le k\le 5, \\
     (2\pi)^{\frac{8k}{81(K - 1)}} C_k + \widehat{D}_k, & 6 \le k.
     \end{cases}
\end{equation}
Note that $A(k)$ additionally depends on $\eta$, $h$, $h_2$ and $t_0$. Assembling \eqref{CkDk_bound_1}, \eqref{Yk_bound}, \eqref{Yk_bound1}
and \eqref{A def},
we obtain that if $1\le m\le M-1$ and $k$ is determined via
the requirement \eqref{mk dependence}, then
\begin{equation}\label{Sm bound}   
    |S_m| <  A(k) \,t^{-\rho(k)}.
\end{equation}

We now divide the argument into two cases: $3 \le k \le 5$ and $k > 5$. In the former case, let $N_k$ denote the number of sums $S_m$ corresponding to each integer $k\ge 3$. 
In other words, $N_k$ is the number of $m$ such that $X_m\in (Y_{k + 1}, Y_k]$.
We observe that $X_m \in (Y_{k + 1}, Y_k]$
if and only if $M_k \le m < M_{k + 1}$, where 
\begin{equation}\label{Mk defn}
M_k := \frac{(1 - \theta_k)\log X_0}{\log h}.
\end{equation}
Therefore, by the bound on $|S_m|$ in \eqref{Sm bound},
\begin{equation}\label{k 3 4 bound}
    \sum_{1\le m<M_6}|S_m| < \sum_{k = 3}^{5}N_k\, A(k)\, t^{-\rho(k)}.
\end{equation}
By counting the maximum number of integers in a continuous interval, for 
$k\ge 3$,
\begin{equation}\label{Nk_estimate}
N_k \le \frac{(\theta_k - \theta_{k + 1})\log X_0}{\log h} + 1 \le \phi_k \log t,
\end{equation}
where
\begin{equation}
   \phi_k=\phi_k(h, t_0):= \frac{\theta_k - \theta_{k + 1}}{2\log h} + \frac{1}{\log t_0}.
\end{equation}
Thus, in view of \eqref{k 3 4 bound} we obtain 
\begin{equation}\label{B0_bound}
 \sum_{1\le m<M_6}|S_m| \le b_0,
\end{equation}
where by \eqref{Nk_estimate} and on explicitly computing the values of $\rho(k)$ for $3\le k\le 5$,  
\begin{equation}\label{B0 def}
\begin{split}
    b_0=b_0(\eta, h, h_2, t_0) &:= A(3) \phi_3\frac{\log t_0}{t_0^{1/30}} + A(4)\phi_4\frac{\log t_0}{t_0^{3/182}} + A(5)\phi_5\frac{\log t_0}{t_0^{7/990}}.
\end{split}
\end{equation}
Here, we used the supposition $t\ge t_0\ge \exp(990/7)$ which ensures 
that $\log t/t^{\rho(k)}$ is decreasing for $t\ge t_0$, with $3\le k\le 5$.

Next, for $k > 5$, we use a tail argument that removes the dependence of $A$ on $k$ in the bound \eqref{Sm bound}. Define 
\begin{equation}\label{A0 def}
A_0=A_0(\eta, h, h_2, t_0) := \sup_{k \ge 6}A(k).
\end{equation}
We will bound $A_0$ explicitly 
later (Section \ref{A0 A2 calc}), after determining $\eta$, $h$, $h_2$ and $t_0$.
Importantly, under our assumptions, it will transpire that $A_0$ is finite. 

In the meantime, we observe that $\rho(k)$ is continuously decreasing in $k\ge 3$ (towards zero), 
as can be seen from computing the derivative of $\rho(k)$ with respect to $k$. 
So, by simply counting the number of terms, 
and noting that as $m$ ranges over $[M_6,M-1]$, 
$k\le r$, we obtain from \eqref{Sm bound},
\begin{equation}\label{A0 bound}
\begin{split}
\sum_{M_6 \le m \le M - 1}|S_m| \le |M - M_6| A_0\, t^{-\rho(r)} \le A_1 \,
\frac{|b_1 - \theta_{r + 1}|}{t^{\rho(r)}}\log t,
\end{split}
\end{equation}
where, from \eqref{M inequality} and \eqref{Mk defn},
\begin{equation*}
A_1=A_1(\eta, h, h_2, t_0) := \frac{1}{2}\cdot\frac{A_0}{\log h}, \qquad b_1 = b_1(h,t_0):= \theta_6 + \frac{2\log h}{\log t_0}.
\end{equation*}
We want to further simplify the bound on the right-side of \eqref{A0 bound}. 
Via \eqref{theta lb 1} and a routine calculation,
\begin{equation}\label{theta r+1 bound}
    \frac{2}{r-1+2^{2-r_0}}\le \theta_{r+1} < b_1,\qquad (r\ge r_0 \ge 6).
\end{equation}
Hence, utilizing the inequality $r\le h_2\log\log t$
 as well as $\log(u)\le u - 1$, valid for any real number $u > 0$, yields
that if $r\ge 6$, then
\begin{equation}\label{A2 bound}
    \begin{split}
        \frac{b_1 - \theta_{r + 1}}{t^{\rho(r)}} \le \exp(b_1 -\theta_{r+1}-1-\rho(r)\log t)
        \le \frac{1}{A_2}\frac{1}{\log t},
    \end{split}
\end{equation}
where (identifying $x$ with $\log t$ below)
\begin{equation}\label{A2 def}
    A_2 = A_2(h, h_2,t_0) := \inf_{x \ge \log t_0}\frac{1}{x}\exp\left(1 - b_1 + \frac{2}{h_2\log x-1+2^{2-r_0}}
    +x\,\rho(h_2\log x)\right).
\end{equation}
In indicating the dependencies of $A_2$, we used that 
$r_0$ is determined by $h_2$ and $t_0$, and $b_1$ further depends on $h$. 
Now, from the definition of $\rho$, and using 
$$2^{h_2\log x}= x^{h_2\log 2},$$ 
we obtain\footnote{By the lower bound on $h_2$ in \eqref{r0_assumption}, we have $x^{h_2\log 2}\ge 64$ for $x\ge \log t_0$. So, $\rho(h_2\log x)$ is well-defined for $x\ge \log t_0$.}
\begin{equation}
   \rho(h_2\log x)= \frac{x^{h_2\log 2}-4}{x^{h_2\log 2}-2}\cdot \frac{1}{(h_2\log x-1)x^{h_2\log 2}+4}.
\end{equation}
Since\footnote{This step is where a slow enough growth on $r$ is utilized, to ensure that 
$(\log t)^{h_2\log 2} = o(\log t)$ as $t\to \infty$, so that the constant $A_2(h_2,t_0)$ defined in \eqref{A2 def} is positive.}
by supposition $0 < h_2 < 1/\log 2$, 
the quantity defining $A_2$ tends to $+\infty$ as $x\to +\infty$.
As this quantity is continuous and positive for $x\ge \log t_0$,
$A_2$ is finite and positive for all admissible $h_2$. From \eqref{A0 bound} and \eqref{A2 bound} we conclude that 
\begin{equation}\label{a1a2 bound tail}
\sum_{M_6 \le m \le M - 1}|S_m| \le \frac{A_1}{A_2}. 
\end{equation}
Lastly, to bound the remaining terms in \eqref{Sm sum bound}, in view of \eqref{M inequality}, we have the rough estimate 
\[
\sum_{m = 1}^{M - 1}\frac{1}{X_m} = \frac{1}{X_0}\sum_{m = 1}^{M - 1}h^{m} = \frac{1}{X_0}\frac{h^{M} - h}{h - 1} < \frac{X_0^{1 - \theta_{r + 1}} - 1}{X_0}\frac{h}{h - 1} < \frac{h}{h - 1}\frac{1}{X_0^{\theta_{r + 1}}}.
\]
Since $r \le h_2\log\log t$ and $\theta_{r + 1} > 2/r$, we have, for $t \ge t_0$,
\begin{equation}\label{extra terms estimate}
\sum_{m = 1}^{M - 1}\frac{1}{X_m} < \frac{h}{h - 1}\exp\left(-\frac{\log (t_0/(2\pi))}{h_2\log\log t_0}\right) =: E(h, h_2, t_0),
\end{equation}
say.

In summary, combining \eqref{Sm sum bound}, \eqref{B0_bound}, \eqref{a1a2 bound tail} and \eqref{extra terms estimate}, the tail sum satisfies the inequality
\begin{equation}\label{tail bound}
\begin{split}
\left|\sum_{X_{M-1} < n \le X_0}\frac{1}{n^{1 +it}}\right| &\le \frac{A_1(\eta, h, h_2,t_0)}{A_2(h, h_2,t_0)} + b_0(\eta, h, h_2, t_0) + E(h, h_2, t_0).
\end{split}
\end{equation}
So, for such $r$, the tail sum is bounded by a constant independent of $t$.

It remains to bound the initial sum over $1\le n \le X_{M - 1}$. We use the triangle inequality followed by the harmonic sum bound \eqref{harmonic bound}, 
which yields
\begin{equation}\label{initial bound 0}
\begin{split}
\left|\sum_{1 \le n \le X_{M - 1}}\frac{1}{n^{1 + it}}\right| &\le \log X_{M - 1} + \gamma + \frac{1}{ 2(X_{M-1}-1)}.
\end{split}
\end{equation}
We use the inequality \eqref{XM_bound} to bound $X_{M - 1}$ from above and below,
so that the right-side of \eqref{initial bound 0} satisfies
\begin{equation}\label{initial bound 1}
\begin{split}
\le \frac{\theta_{r + 1}}{2}\log \frac{t}{2\pi} + \log h + \gamma + \frac{1}{2(X_0^{\theta_{r + 1}}-1)}.
\end{split}
\end{equation}
Since $\theta_{r+1} < 2/(r-1)$ by \eqref{theta lb 1}, the first term in \eqref{initial bound 1} is bounded by
\begin{equation}\label{theta bound}
\frac{\log (t/(2\pi))}{r - 1} < \frac{\log t - \log 2\pi}{h_2\log\log t - 2}\le \frac{1 - \frac{\log 2\pi}{\log t_0}}{h_2 - \frac{2}{\log\log t_0}}\frac{\log t}{\log\log t},
\end{equation}
where the last inequality is due to the fact that $(1 - \frac{\log 2\pi}{x})/(h_2 - \frac{2}{\log x})$ is decreasing for $x \ge 990/7$ if $h_2 < 1/\log 2$. 

In addition, we use \eqref{const term bound 457} to bound the last term. Putting this together with \eqref{initial bound 0}, \eqref{initial bound 1} and \eqref{theta bound}, we obtain
\begin{equation}\label{initial bound}
\left|\sum_{1 \le n \le X_{M - 1}}\frac{1}{n^{1 + it}}\right| 
< A_3\frac{\log t}{\log\log t},
\end{equation}
where 
\begin{equation*}\label{A3 def}
\begin{split}
     A_3= A_3(h, h_2, t_0)&:= \frac{1 - \frac{\log 2\pi}{\log t_0}}{h_2 - \frac{2}{\log\log t_0}} + \left(\log h + \gamma + 
\frac{1}{2(\kappa(h_2,t_0)-1)}\right)\frac{\log\log t_0}{\log t_0}.
\end{split}
\end{equation*}

Lastly, the bound on  
$\zeta(1 + it)$ in \eqref{zeta bound 1} 
carries an extra term $g(t)/\sqrt{2\pi}+\mathcal{R}(t)$. 
Using the definitions in Theorem~\ref{patel theorem}, 
this term is easily bounded by $g(t_0)/\sqrt{2\pi} + \mathcal{R}(t_0)$. Therefore, combining this with \eqref{tail bound} and \eqref{initial bound}, 
 we see that 
\[
|\zeta(1 + it)| \le \left|\sum_{1 \le n \le n_1}\frac{1}{n^{1+it}}\right| + \frac{g(t_0)}{\sqrt{2\pi}} + \mathcal{R}(t_0) \le 
A_4\frac{\log t}{\log\log t},
\]
for $t \ge t_0$, where 
\[
A_4= A_4(\eta, h, h_2, t_0):=  A_3+ \left(\frac{A_1}{A_2} + b_0 + E + \frac{g(t_0)}{\sqrt{2\pi}} +\mathcal{R}(t_0)\right)\frac{\log\log t_0}{\log t_0}.
\]

We choose the following values for our free parameters,
which are suggested by numerical experimentation.
\begin{equation}\label{parameter choices}
\eta = 1.2364,\qquad h = 1.0224,\qquad h_2 = 0.85532,\qquad t_0 = \exp(3381).
\end{equation}
With these choices, 
our preconditions on $t_0$ and $h_2$ hold,
and the conditions $h >1$ and $r_0 \ge 6$ also hold (in fact, $r_0 = 6$).
Additionally, in Section \ref{A0 A2 calc}, we show that
\begin{align*}
A_0 &= 0.0976483443\ldots,\\
A_2 &= 0.0618341521\ldots.
\end{align*}
Therefore, we can explicitly compute 
\begin{align*}
b_0 &\le 4.629\cdot 10^{-8},\\
A_1 &\le 2.2039724975\ldots,\\
A_3 &\le 1.6420608139\ldots,\\
A_4 &\le 1.7301296329\ldots.
\end{align*}
In summary, the 
assertion of Theorem~\ref{main theorem} holds when $t\ge t_0=\exp(3381)$.

For smaller $t\le t_0$ we use the bound in Patel \cite{patel_explicit_2022}
and Proposition~\ref{third_deriv_bound}. Specifically, 
the bound in \cite[Theorem 1.1]{patel_explicit_2022} yields

\begin{equation*}
\begin{split}
&|\zeta(1 + it)| \le \frac{1}{5}\log t + 44.02 \le 1.731\frac{\log t}{\log \log t},\qquad \exp(662) \le t \le \exp(3381),\\
&|\zeta(1 + it)| \le \frac{1}{2}\log t + 1.93 \le 1.731\frac{\log t}{\log \log t},\qquad 3 \le t \le \exp(16).
\end{split}
\end{equation*}
This leaves the gap interval 
$\exp(16)\le t\le \exp(662)$ 
for which we invoke Proposition~\ref{third_deriv_bound}.
Since the gap interval falls in the range 
of $t$ where Proposition~\ref{third_deriv_bound} is applicable, 
our target bound is verified there as well.

\section{Bounding $A_0$ and $A_2$}\label{A0 A2 calc}

We derive an upper bound on $A_0(\eta, h, h_2, t_0)$, where the parameter values are given in \eqref{parameter choices}.  
For the convenience of the reader, we recall the definition
$A_0 := \sup_{k \ge 6}A(k)$, where for $k\ge 6$,
$$A(k) := (2\pi)^{k/(81K - 81)} C_k(\eta, h) + \widehat{D}_k(\eta, h, h_2, t_0).$$
With our choice of free parameter values in \eqref{parameter choices}, we 
numerically compute the following list of $A(k)$.
\begin{align*}
A(6) &= 0.0976483443\ldots,\qquad
A(7) =  0.0914388403\ldots,\\
A(8) &= 0.0882612019\ldots,\qquad
A(9) =  0.0865130096\ldots,\\
A(10) &=0.0855306439\ldots.
\end{align*}

The values of $C_k$ and $D_k$ required for the above computation 
are obtained using the formulas in Lemma~\ref{CD_deriv_test}.
In turn, evaluating $C_k(\eta, h)$ and $D_k(\eta, h)$ relies
on  evaluating $A_k(\eta,h^k)$ and $B_k(\eta)$.
When $k=3$,  $A_k$ and $B_k$ are computed directly
 using their definitions in Lemma~\ref{third_deriv_test}. 
 When $k\ge 4$, the recursive formulas
 in Lemma~\ref{kth_deriv_test} are used to compute them.
Let us examine these recursive formulas. 

First, the following quantity, which appears in
the formula for $A_{k+1}(\eta,\omega)$ in \eqref{s1_Ak_recursive_defn}, 
satisfies
\begin{equation}\label{K ratio bound}
    \max_{k\ge 10} \frac{2^{19/12}(K - 1)}{\sqrt{(2K - 1)(4K - 3)}} = \frac{2^{19/12}\cdot 511}{\sqrt{2092035}}=:\mu_{10}.
\end{equation}
Here, $K=2^{k-1}$ and the maximum occurs when $k=10$.
Another quanitity appearing in the formula for $A_{k+1}(\eta,\omega)$ is  $\delta_k$, 
which is clearly seen to decrease with $k$.
Noting that $h>1$, hence $\omega=h^{k+1} > 1$ 
for any $k\ge 0$, we deduce that 
\[
A_{k + 1}(\eta,\omega) \le \delta_{10}\left(1 + \mu_{10}A_k(\eta,\omega)^{1/2}\right).
\]
So, we are led to consider the discrete map 
$$x_{n + 1} = \delta_{10}\left(1 + \mu_{10} x_n^{1/2}\right).$$ 
By a routine calculation, this map has a single fixed point at
\begin{equation}
x^* = \left(\frac{\mu_{10}\delta_{10}}{2} + \sqrt{\frac{\mu_{10}^2\delta_{10}^2}{4} + \delta_{10}}\right)^2 = 2.7600429449\ldots,
\end{equation}
which is a stable point. Since $A_{k + 1}(\eta,\omega) \le x_{k + 1}$ for $k\ge 10$ 
and, by a direct numerical computation, $A_{10}(\eta, \omega) \le x^*$, 
it follows that 
\begin{equation}\label{Ak bound}
    A_{k}(\eta, h^k) \le x^*,\qquad (k \ge 10). 
\end{equation}

The analysis of $B_{k+1}$ for $k\ge 10$ proceeds similarly. The recursive 
formula for $B_{k+1}$ contains
the following quantity satisfying
\begin{equation}
\max_{k\ge 10}\frac{2^{3/2}(K - 1)}{\sqrt{(2K - 3)(4K - 5)}} \le \frac{2^{3/2}\cdot 511}{3\sqrt{231767}}=:\nu_{10}.
\end{equation}
So,
$B_{k+1}\le \delta_{10} \nu_{10} B_k^{1/2}$ for $k \ge 10$.
Consider the discrete map $y_{n + 1} = \delta_{10} \nu_{10} y_n^{1/2}$. 
This map has a single fixed point, namely, 
$$y^* = (\delta_{10} \nu_{10})^2 = 1.0023463404\ldots,$$ 
which is also stable.
Since $B_{k+1} \le y_{k+1}$ for $k\ge 10$ and, by a direct numerical computation, 
$B_{10} \le y^*$, it follows that 
\begin{equation}\label{Bk bound}
    B_k(\eta)\le y^*, \qquad (k \ge 10).
\end{equation}
 
Furthermore, we numerically verify that for $k \ge 10$,
\[
h^{2k/K - k / (2K - 2)}(h - 1)\left(\frac{(k - 1)!}{2\pi}\right)^{\frac{1}{2K - 2}} < 0.023,
\]
\[
h^{k/(2K - 2)}(h - 1)^{1 - 2/K}\left(\frac{2\pi}{(k - 1)!}\right)^{\frac{1}{2K - 2}} < 0.023.
\]
Combining this with \eqref{Ak bound} and \eqref{Bk bound}, 
and recalling the definitions in Lemma~\ref{CD_deriv_test} and \eqref{D hat defn}, yields 
the inequalities $C_k(\eta,h) < 0.023 x^*$ and  $\widehat{D}_k(\eta,h) < 0.024 y^*$
for $k\ge 10$. 
Consequently, $A(k) < 0.09$ for $k \ge 10$. 

Furthermore, in view of the values 
of $A(k)$ listed at the beginning, we also see that 
$A(k) \le A(6) = 0.09764\ldots$ for $6\le k\le 10$. 
We therefore conclude $A_0 = A(6)$.  

For the computation related to $A_2$, we shall use a derivative calculation to show that the function
\begin{equation}\label{A2 exponent}
    a_2(u) := 1-b_1+ e^u\,\rho(h_2 u)+\frac{2}{h_2u-1+2^{2-r_0}} -u
\end{equation}
is increasing in $u\ge u_0:=\log \log t_0$.
So, on identifying $u$ with $\log x$ in the definition of $A_2$ in \eqref{A2 def}, 
the value of $A_2$ is $\exp(a_2(u_0))$.

To this end,  we first calculate 
$$\frac{d}{du} e^u\,\rho(h_2 u) = e^u\rho(h_2 u)\left(1+\frac{d}{du} \log \rho (h_2u)\right), 
\qquad \frac{d}{du} \log \rho (h_2u)=-h_2 \tilde{a}_2(h_2 u),$$
where 
$$\tilde{a}_2(v) = \frac{(v-1)\log 2+1-2^{-v}(8(v-1)\log 2+6)+4^{-v}(8(v-2)\log 2+8)}
{(v-1+2^{2-v})(1-2^{1-v})(1-2^{2-v})}.$$
Note that $v$ is identified with $h_2u$, so we need only consider $v\ge h_2u_0=:v_0$. It is easy to see that 
$$\sup_{v\ge v_0} \tilde{a}_2(v) < \sup_{v\ge v_0} \frac{(v-1)\log 2+1}{(v-1+2^{2-v_0})(1-2^{1-v_0})(1-2^{2-v_0})}\le \log 2+0.2,$$
where the supremum occurs at $v=v_0$.
Hence, the derivative 
of $ e^u\,\rho(h_2 u)$ is at least $e^u\rho(h_2 u) (1-(\log 2+0.2)h_2)$ for $u\ge u_0$. 
Importantly, $1-(\log 2+0.2)h_2$ is positive. In comparison, the derivative of the remaining terms in the formula for $a_2(u)$ is
$$\frac{d}{du} \left(\frac{2}{h_2u-1+2^{2-r_0}} -u\right) = -\frac{2h_2}{(h_2 u -1+2^{2-r_0})^2}-1 \ge -1.04731\ldots,$$
since $u\ge u_0$ and $r_0 = 6$. Since 
\begin{equation}\label{condition1}
e^u\rho(h_2 u)(1-(\log 2+0.2)h_2) \ge 1.06112\ldots
\end{equation}
for $u\ge u_0$, the derivative of $a_2(u)$ is positive for $u\ge u_0$, proving our claim.

\section{Proof of Proposition~\ref{third_deriv_bound}}\label{Proposition proof}

The calculation is similar to that in \cite{hiary_explicit_2016} on the $1/2$-line, so we shall be brief in some details. 
We divide the main sum in \eqref{zeta bound 1} into pieces of length $\approx t^{1/3}$,
then apply the van der Corput Lemma~\ref{corput lemma} to each piece other than an initial segment, which
is bounded  by the triangle inequality.

Let $t_1$ be a positive number to be chosen later,
subject to $t_1\ge \exp(12)$. Suppose that $t\ge t_1$, unless otherwise stated, and let $n_1$ be as defined in \eqref{zeta bound 1}. 
Let $P =\lceil t^{1/3}\rceil$, $Q = \lfloor n_1/P\rfloor$, and let 
$q_0$ a positive integer to be chosen later, subject to $q_0\le Q$. 
Thus, we have
\begin{equation}\label{zeta bound 2}
    \begin{split}
        \left|\sum_{n=1}^{n_1} \frac{1}{n^{1+it}}\right|
&\le \sum_{n=1}^{q_0 P-1} \frac{1}{n}+
\sum_{q=q_0}^{Q-1} \left|\sum_{n=qP}^{(q+1)P-1} \frac{1}{n^{1+it}} \right|
+\left|\sum_{QP}^{n_1} \frac{1}{n^{1+it}}\right|.
    \end{split}
\end{equation}

The harmonic sum bound \eqref{harmonic bound} yields, subject to $q_0P> 1$, 
\begin{equation}\label{initial sum bound}
    \sum_{n=1}^{q_0 P-1} \frac{1}{n}\le \log (q_0 P-1)+\gamma+\frac{1}{2(q_0P-1)}.
\end{equation}

For $q_0\le q\le Q$, let $N = qP$ and let $N'$ be an integer 
such that  $N \le N' < N+P$. Partial summation gives
\begin{align}
\left|\sum_{n=N}^{N'}\frac{1}{n^{1+it}}\right| &\le \frac{1}{N} \max_{1 \le L \le P} |S_N(L)|,
\end{align}
where 
\[
S_N(L) := \sum_{n = 0}^{L - 1} \exp(2\pi i f(n)) \qquad
\text{and} \qquad f(x) = \frac{t}{2\pi}\log (N + x).
\]
We estimate $S_N(L)$ using Lemma \ref{corput lemma} with $r = q$, $K = P$ and 
\[
K_0 = t_1^{1/3} \le P. 
\]
This gives, for any positive $\eta$,
\begin{align*}
\frac{1}{N}\max_{1 \le L \le P} |S_N(L)|  
&\le \frac{1}{q}\sqrt{\frac{\alpha}{W^{1/3}} + \frac{\beta W^{1/3}}{P} + \frac{\alpha\eta}{P} + \frac{\eta\beta W^{2/3}}{P^2}}.
\end{align*}
where $W$, $\lambda$, $\alpha = \alpha(W, \lambda)$, $\beta = \beta(W)$ are defined in Lemma \ref{corput lemma}. 
We estimate the (nonnegative) terms under the square-root as follows. 
\begin{align*}
        &\frac{\alpha}{W^{1/3}} < \frac{\alpha}{\pi^{1/3}q}, 
        &\frac{\beta W^{1/3}}{P} &\le  \frac{\beta \pi^{1/3} (1 + 1/q_0) q}{t^{1/3}},\\
        &\frac{\alpha\eta}{P} \le \frac{\alpha\eta}{t^{1/3}}, 
        &\frac{\eta\beta W^{2/3}}{P^2} &\le \frac{\eta \beta \pi^{2/3} (1 + 1/q_0)^2 q^2}{t^{2/3}}.
\end{align*}
Isolating the first term  using the inequality $\sqrt{x + y} \le \sqrt{x} + \sqrt{y}$, valid for nonnegative $x$ and $y$, 
and using the inequality 
$q\le Q \le t^{1/6}/(2\pi)^{1/2}$
to bound $q$ in the remaining terms, therefore gives
\begin{equation} \label{SN_bound 0}
\frac{1}{N}\max_{1 \le L \le P} |S_N(L)|  
\le \frac{\alpha(W, \lambda)^{1/2}}{\pi^{1/6}q^{3/2}} + \frac{B(W, \lambda)}{q t^{1/12}},
\end{equation}
where, since $t\ge t_1$,
\[
B(W, \lambda) := \sqrt{\frac{1 + 1/q_0}{2^{1/2}\pi^{1/6}}\beta(W) + \frac{\eta}{t_1^{1/6}}\alpha (W, \lambda) + \frac{\eta(1 + 1/q_0)^{2}}{2\pi^{1/3} t_1^{1/6}}\beta(W)}.
\]
We also note that for all $q\ge q_0$ we have $W\ge W_0$ and $\lambda\le \lambda_0$, where
\begin{equation*}
    \begin{split}
        W_0 := \pi(q_0 + 1)^3,\qquad & \lambda_0 := \left(1 + \frac{1}{q_0}\right)^3.
    \end{split}
\end{equation*}
Therefore, for all $q$ under consideration, $\alpha(W) \le \alpha(W_0, \lambda_0)$ and $\beta(W) \le \beta(W_0)$.
It follows that for $q\ge q_0$, 
\begin{equation} \label{SN_bound}
\frac{1}{N}\max_{1 \le L \le P} |S_N(L)|  
\le \frac{\alpha_0^{1/2}}{\pi^{1/6}q^{3/2}} + \frac{B_0}{q t^{1/12}}.
\end{equation}
where $\alpha_0 := \alpha(W_0, \lambda_0)$ and $B_0 := B(W_0, \lambda_0)$.

Put together, we have
\begin{equation}\label{tail sum bound}
    \begin{split}
        \left|\sum_{n=q_0P}^{n_1} \frac{1}{n^{1+it}}\right|\le & 
        \frac{\alpha_0^{1/2}}{\pi^{1/6}}\sum_{q=q_0}^Q \frac{1}{q^{3/2}} + 
\frac{B_0}{t^{1/12}}\sum_{q=q_0}^Q \frac{1}{q}.
    \end{split}
\end{equation}
Moreover, from the argument in \cite[(3.1)]{hiary_improved_2022}, and using the bound $Q \le t^{1/6}/(2\pi)^{1/2}$ once again,
\begin{equation}\label{tail sum bound 1}
    \begin{split}
        &\sum_{q = q_0}^Q\frac{1}{q^{3/2}} \le \int_{q_0 - 1/2}^{Q + 1/2}\frac{\text{d}x}{x^{3/2}} < \frac{2}{\sqrt{q_0 - 1/2}},\\
        \sum_{q = q_0}^Q \frac{1}{q} &\le \int_{q_0 - 1/2}^{Q + 1/2}\frac{\text{d}x}{x} <
        \frac{1}{6}\log t + \log \tau_1 - \log\left(q_0 - 1/2\right),
    \end{split}
\end{equation}
where 
$$\tau_1 := \frac{1}{(2\pi)^{1/2}} + \frac{1}{2t_1^{1/6}}.$$
Therefore, observing that $\frac{1}{6}\frac{\log t}{t^{1/12}}$ is decreasing for $t\ge t_1\ge \exp(12)$,
and combining this with the initial sum bound \eqref{initial sum bound}, we obtain
\begin{align*}
\left|\sum_{n = 1}^{n_1}\frac{1}{n^{1+it}}\right| &< \log \left(q_0 P-1\right)+\gamma+\frac{1}{2(q_0P-1)} + \frac{\alpha_0^{1/2}}{\pi^{1/6}}\frac{2}{\sqrt{q_0 - 1/2}} \\
&\quad+ \frac{B_0}{t_1^{1/12}}\left(\frac{1}{6}\log t_1 + \log\tau_1 - \log\left(q_0 - 1/2\right)\right).
\end{align*}
We now estimate
\[
\log \left(q_0 P - 1\right) \le \log\left(q_0t^{1/3} + q_0 - 1\right) \le \frac{1}{3}\log t + \log\left(q_0  + (q_0 - 1)t_1^{-1/3}\right),
\]
which follows using $P< t^{1/3}+1$. Furthermore, 
\[
\frac{1}{2(q_0 P - 1)} \le \frac{1}{2\big(q_0t_1^{1/3} - 1\big)}.
\]
This yields 
\begin{equation}\label{zeta k=3 bound 1}
    |\zeta(1 + it)| \le \frac{1}{3} \log t + B_1(q_0)+ \frac{g(t_1)}{\sqrt{2\pi}} + 
\mathcal{R}(t_1),
\end{equation}
where
\begin{align*}
B_1(q_0) := &\, \gamma+ \log\left(q_0  + (q_0 - 1)t_1^{-1/3}\right)  + \frac{1}{2\big(q_0t_1^{1/3} - 1\big)}\\
&+\frac{\alpha_0^{1/2}}{\pi^{1/6}}\frac{2}{\sqrt{q_0 - 1/2}} + \frac{B_0}{t_1^{1/12}}\left(\frac{1}{6}\log t_1 + \log\tau_1 - \log\left(q_0 - 1/2\right)\right).
\end{align*}
Choosing $t_1 = \exp(16)$, $\eta = 1.267$ and $q_0 = 5$, we obtain
\begin{equation} \label{zeta k=3 bound 2}
|\zeta(1 + it)| \le \frac{1}{3}\log t + 4.66,\qquad t \ge \exp(16).
\end{equation}
This implies Theorem~\ref{main theorem} in the range $\exp(16) \le t \le \exp(88)$.

Consider now the second part of Proposition \ref{third_deriv_bound}. In this critical region, we make use of a few additional tools to sharpen our estimates. In particular, we employ a different method of bounding the second sum appearing in Theorem \ref{patel theorem}, leading to an estimate of size $O(t^{-1/12})$ instead of $O(1)$. Also, we employ the $k$-th derivative tests used in the proof of Theorem \ref{main theorem}, however this time we introduce additional scaling parameters to further optimise the switching points between derivative tests of different order. 

To this end, let $t_2$ be a real number
to be chosen later, subject to $t_2 \ge \exp(182/3)$.
Assume $t \ge t_2$.
Recall from \eqref{tail sum bound} and \eqref{tail sum bound 1} we showed for $q_0 \ge 1$, $P = \lceil t^{1/3}\rceil$ and $t \ge t_2$ that the tail sum
satisfies
\begin{equation}\label{tail sum bound 2}
\left|\sum_{n=q_0P}^{n_1} \frac{1}{n^{1+it}}\right|\le 
\frac{\alpha_0^{1/2}}{\pi^{1/6}}\frac{2}{\sqrt{q_0 - 1/2}} + \frac{B_0}{t_2^{1/12}}\left(\frac{1}{6}\log t_2 + \log\tau_1 - \log\left(q_0 - 1/2\right)\right).
\end{equation}
As for the initial sum over $1\le n<  q_0P$, 
instead of bounding it using the triangle inequality throughout as in \eqref{initial sum bound}, 
we will refine our estimate by using $k$-th order van der Corput lemmas with $k = 4$ and $5$
in the subintervals $(Z_5, Z_4]$ and $(Z_6, Z_5]$ respectively, where
\begin{align*}
    Z_4:=q_0 P - 1, \qquad Z_5 := \mu_5 t^{\theta_5/2}, \qquad Z_6 := \mu_6 t^{\theta_6/2},
\end{align*}
with $\theta_5,\theta_6$ the same as in the proof of Theorem \ref{main theorem} 
and $\mu_5,\mu_6$ real positive numbers to be chosen later, subject to
\begin{equation}\label{mu_condition}
1 < \mu_6 t_2^{\theta_6/2} \le \mu_5 t_2^{\theta_5/2} \le q_0\lceil t_2^{1/3}\rceil - 1.
\end{equation}
We remark that $Z_5,Z_6$ will serve a similar purpose as $Y_k$ in the proof of Theorem \ref{main theorem}, 
whereas the new parameters $\mu_5,\mu_6$ will allow us to fine-tune our estimates in this critical region. 

Let $h' > 1$ be a real number to be chosen later, and let $X_m' = (h')^{-m}Z_4$ for integer $m$. Define
\[
M' := \left\lceil\frac{\log (Z_4/Z_6)}{\log h'}\right\rceil.
\]
The definition of $M'$ guarantees $X'_{M'} \le Z_6 < X'_{M' - 1}$. Next, define the subsum
$$S_m':= \sum_{X_{m}' + 1 < n \le X_{m - 1}'}\frac{1}{n^{1 + it}},$$
so that, similarly to \eqref{Sm sum bound}, 
\[
\left|\sum_{X'_{M' - 1} < n \le Z_4}\frac{1}{n^{1 + it}}\right| \le \sum_{1 \le m \le M' - 1}\left(\frac{1}{X_m'} + |S_m'|\right).
\]
By a similar calculation 
as in \eqref{XM_XM1_ratio0}, for each $1 \le m \le M' - 1$,
\[
\frac{\lfloor X_{m - 1}'\rfloor}{\lfloor X_{m}'\rfloor + 1} \le h'.
\]
Let $k=4$ or $5$. Following the argument of Theorem \ref{main theorem}, if $Z_{k + 1} < X_{m}' \le Z_{k}$, then
\begin{align*}
\left|S_m'\right|
&\le C_k(\eta, h') (\lfloor X_{m}'\rfloor + 1)^{-k/(2K - 2)}t^{1/(2K - 2)} \\
&\qquad\qquad  + D_k(\eta, h') (\lfloor X_m' \rfloor + 1)^{k/(2K - 2)- 2/K}t^{-1/(2K - 2)}\\
&\le C_k(\eta, h') Z_{k + 1}^{-k/(2K - 2)}t^{1/(2K - 2)} + D_k(\eta, h') (Z_k + 1)^{k/(2K - 2)- 2/K}t^{-1/(2K - 2)}.
\end{align*}
We write
$\mu_4:=q_0 + (q_0 - 1)t_2^{-1/3}$ so that $Z_4 \le \mu_4 t^{1/3}$. Since $k / (2K - 2) - 2/K \ge 0$, for $t \ge t_2$,
$$
|S_m'| \le 
A_k'(t) := \begin{cases}
\displaystyle \mu_{5}^{-2/7}C_4 t^{-3/182} + (\mu_4 + t_2^{-1/3})^{1/28}D_4 t^{-5/84}, & k = 4,\\
\displaystyle \mu_{6}^{-1/6}C_5 t^{-7/990} + (\mu_5 + t_2^{-4/13})^{1/24}D_5 t^{-4/195}, & k = 5.
\end{cases}
$$
Let $N_k'$ be the number of $X_{m}'$ satisfying $Z_{k + 1} < X_{m}' \le Z_{k}$. Then 
\begin{align*}
N_k' &\le \frac{\log Z_{k} - \log Z_{k + 1}}{\log h'} + 1,\qquad 
\text{therefore}\qquad N_k' \le B'_k\frac{\log t}{\log h'},
\end{align*}
where
\[
B'_k := \begin{cases}
\displaystyle\frac{1}{39} + \frac{\max\{0, \log \mu_4 - \log \mu_{5} + \log h'\}}{\log t_2},&k = 4,\\
\, &\\
\displaystyle\frac{28}{429} + \frac{\max\{0, \log \mu_{5} - \log \mu_{6} + \log h'\}}{\log t_2},&k = 5.
\end{cases}
\]
Put together, we have
\[
\sum_{Z_{k + 1} < X'_{m} \le Z_{k}}|S_m'| \le \frac{A'_k(t) B'_k}{\log h'}\log t.
\]

For $t \ge t_2$, the function $A_4'(t)\log t$ is decreasing, so $A_4'(t)\log t \le A_4'(t_2)\log t_2$. For $k = 5$, we use 
\[
t^{-7/990}\log t \le \frac{990}{7e}\qquad\text{and}\qquad t^{-4/195}\log t \le t_2^{-4/195}\log t_2,\qquad (t \ge t_2).
\]
Therefore, 
\begin{equation}\label{4th_deriv_test_bound_2}
\begin{split}
\sum_{1 \le m \le M' - 1}|S'_m| &\le \frac{1}{\log h'}\bigg(A'_4(t_2) B'_4\log t_2 + \frac{990}{7e}\frac{B'_5 C_5}{\mu_{6}^{1/6}} \\
&\qquad\qquad + \left(\mu_5 + t_2^{-4/13}\right)^{1/24}B'_5D_5t_2^{-4/195}\log t_2\bigg).
\end{split}
\end{equation}
Next, similarly to \eqref{extra terms estimate}, we obtain
\begin{equation}\label{fourth deriv extra terms}
\sum_{1 \le m \le M' - 1}\frac{1}{X_m'} < \frac{h'}{h' - 1}\frac{1}{Z_6} \le \frac{h'}{(h' - 1)\mu_6 t_2^{\theta_6/2}},\qquad (t \ge t_2).
\end{equation}
In the remaining range 
$1\le n \le X'_{M' - 1}$, 
we use that
$h' Z_6 \ge X'_{M' - 1}$. 
Combined with the triangle inequality and \eqref{harmonic bound}, this yields 
\begin{equation}\label{4th_deriv_test_bound_3}
\begin{split}
\left|\sum_{1 \le n \le \lfloor X_{M' - 1}\rfloor}\frac{1}{n^{1+it}}\right| &\le \sum_{1 \le n \le \lfloor h'Z_6\rfloor}\frac{1}{n}
\le \log h' Z_6 + \gamma + \frac{1}{2(h'Z_6 - 1)} \\
&\le \frac{\theta_6}{2}\log t + \log \mu_6 + \log h' + \gamma + \frac{1}{2\big(h'\mu_6 t_2^{\theta_6/2} - 1\big)}.
\end{split}
\end{equation}

Next, consider the second sum 
$\sum_{1 \le n \le n_1}n^{it}$ in Theorem~\ref{patel theorem},
which was just bounded by $n_1$ in our previous calculations.
Let $q$, $q_0$, $P$, $Q$ and $N$ be as defined in the proof of Proposition \ref{third_deriv_bound}, and suppose as before that $t \ge t_2 \ge \exp(182/3)$. 
Using $P \le t^{1/3} + 1$,
the estimate \eqref{SN_bound} gives for $q_0 \le q \le Q$,
\[
\max_{1 \le L \le P} |S_N(L)|  \le \left(t^{1/3} + 1\right)\left(\frac{\alpha_0^{1/2}}{\pi^{1/6}q^{1/2}} + \frac{B_0}{t^{1/12}}\right).
\]
We combine this with the inequality
\[
\sum_{q = q_0}^{Q}\frac{1}{q^{1/2}} \le \int_{q_0 - 1}^{Q}\frac{\text{d}x}{x^{1/2}} < 2Q^{1/2},
\]
as well as bound the sum over $1\le n < q_0 P$ by the number of terms in the sum, which gives
\begin{align*}
\left|\sum_{n = 1}^{n_1}n^{it}\right| &< q_0 P - 1 + \left(t^{1/3} + 1\right)\left(\frac{2\alpha_0^{1/2}Q^{1/2}}{\pi^{1/6}} + \frac{B_0Q}{t^{1/12}}\right).
\end{align*}
Since $Q \le t^{1/6}/\sqrt{2\pi}$, we obtain for $t\ge t_2$,
\begin{equation}\label{second sum estimate}
    \left|\sum_{n = 1}^{n_1}n^{it}\right| < C\, t^{5/12},
\end{equation}
where 
\begin{align*}
C(q_0, \eta, t_2) &:= \left(1 + t_2^{-1/3}\right)\left(\frac{q_0}{t_2^{1/12}}+\frac{2^{3/4}\alpha_0^{1/2}}{\pi^{5/12}} + \frac{B_0}{\sqrt{2\pi}}\right).
\end{align*}
Therefore, after multiplying by $g(t)t^{-1/2}$, we obtain a sharper bound 
on the contribution of the second sum in Theorem~\ref{patel theorem}. Namely, the bound
$C g(t_2)\,t^{-1/12}$, which replaces the original estimate $g(t_2)/\sqrt{2\pi}$.

Lastly, combining \eqref{tail sum bound 2}, \eqref{4th_deriv_test_bound_2}, \eqref{fourth deriv extra terms}, \eqref{4th_deriv_test_bound_3} and \eqref{second sum estimate}, and noting that $\theta_6 = 16/33$,
we have for $t \ge t_2$ 
\[
|\zeta(1 + it)| \le \frac{8}{33}\log t + B_3,
\]
where
\begin{align*}
B_3 :=&\, \frac{1}{\log h'}\left(A'_4(t_2) B'_4\log t_2 + \frac{990}{7e}\frac{B'_5 C_5}{\mu_{6}^{1/6}}
 + (\mu_5 + t_2^{-4/13})^{1/24}B'_5D_5t_2^{-4/195}\log t_2\right)\\
&\qquad + \frac{h'}{(h' - 1)\mu_6 t_2^{\theta_6/2}} + \frac{\alpha_0(W_0, \lambda_0)^{1/2}}{\pi^{1/6}}\frac{2}{\sqrt{q_0 - 1/2}}\\
&\qquad + \frac{B_0(W_0, \lambda_0)}{t_2^{1/12}}\left(\frac{1}{6}\log t_2 + \log\tau_1 - \log\left(q_0 - 1/2\right)\right)\\
&\qquad + \log \mu_6 + \log h'+\gamma + \frac{1}{2\big(h'\mu_6 t_2^{\theta_6/2} - 1\big)} + C(q_0, \eta, t_2)\frac{g(t_2)}{t_2^{1/12}} + \mathcal{R}(t_2).
\end{align*}
We choose 
\begin{alignat*}{3}
\eta &= 1.3348,\qquad\qquad  
&&h' = 1.056,\qquad\qquad  
&&q_0 = 46,\\
\mu_5 &= 48.575,\qquad\qquad  
&&\mu_6 = 51.296,\qquad\qquad  
&&t_2 = \exp(90).
\end{alignat*}
Verifying that \eqref{mu_condition} holds, we calculate $B_3 \le 12.45$, i.e. that 
\[
|\zeta(1 + it)| \le \frac{8}{33}\log t + 12.45,\qquad t \ge \exp(88).
\]
This implies Theorem \ref{main theorem} for $\exp(88) \le t \le \exp(662)$.

\section{Proof of Theorem~\ref{main theorem 1}}\label{Theorem 2 proof}

We begin by proving some bounds on $\zeta(s)$. To help make our notation more intuitive, we will allow reuse of variable names from previous sections. The reader should keep in mind these variables mean different things in the confines of this section. Since this section is independent from Section {\ref{Theorem 1 proof}, Section {\ref{A0 A2 calc}} and Section {\ref{Proposition proof}}, this will hopefully cause no confusion.

\begin{lemma}\label{zeta_log_bound}
Let $C > 0$ and $t_0 \ge e^e$ be constants. 
Suppose $t \ge t_0$. Define
$$\sigma_t:=1 - \frac{C(\log\log t)^2}{\log t}.$$
Additionally, suppose that $C$ satisfies 
\begin{equation}\label{C assumption}
\frac{C(\log\log t_e)^2}{\log t_e} \le \frac{1}{2},\qquad \text{where}\qquad
t_e := e^{e^2}.
\end{equation}
For each $t\ge t_0$, if $\sigma_t\le \sigma \le 2$, then
\begin{equation}\label{zeta_AB}
|\zeta(\sigma + it)| \le A\log^Bt\qquad \text{where}\qquad
A = 76.2,\qquad B = \frac{2}{3} + \frac{71.2\, C^{3/2}}{e^2}.
\end{equation}
\end{lemma}
\begin{proof}
From Ford \cite{ford_vinogradov_2002}, if $1/2 \le \sigma \le 1$, then
\begin{equation}\label{ford bound}
    |\zeta(\sigma + it)| \le 76.2\, t^{4.45(1 - \sigma)^{3/2}}\log^{2/3}t.
\end{equation}
As $\sigma_t\ge 1/2$, a consequence of the assumption \eqref{C assumption} 
and considering that $(\log \log t)^{2}/\log t$ reaches its maximum at $t = t_e$,
we may use \eqref{ford bound} for any $\sigma \in [\sigma_t, 1]$.  

In view of this, since $(1-\sigma)^{3/2}$ is monotonically decreasing with $\sigma$,
if $\sigma_t \le \sigma \le 1$, then \eqref{ford bound} gives
\begin{align}\label{expr_exp}
|\zeta(\sigma + it)| \le 76.2 \exp\left(4.45(1-\sigma_t)^{3/2}\log t + \frac{2}{3}\log\log t\right).
\end{align}
Or, written differently,
\begin{equation}\label{Btilde bound}
    |\zeta(\sigma + it)| \le A \exp\left(\tilde{B}\log\log t\right),
\end{equation}
where $A$ is defined as in \eqref{zeta_AB} and
\begin{equation*}
\tilde{B} := \frac{2}{3}+ 4.45\, C^{3/2}\,\frac{(\log\log t)^2}{\sqrt{\log t}}.
\end{equation*}
As $(\log \log t)^{2}/\sqrt{\log t}$
reaches its maximum at $t = e^{e^4}$, attaining the value $16/e^2$, we obtain
after a simple computation that
$\tilde{B} \le B$, with $B$ defined as in \eqref{zeta_AB}.
The desired result therefore follows for $\sigma_t\le \sigma\le 1$
on replacing $\tilde{B}$ with $B$ in \eqref{Btilde bound}.

To show that the result holds for $1 \le \sigma \le 2$ as well, we use the Phragm\'en--Lindel\"of Principle on the holomorphic function 
$$f(s) = (s - 1)\zeta(s).$$ 
To this end, on the $1$-line we have
\[
|f(1 + it)| = |t|\,| \zeta(1 + it)| \le 62.6\, |2 + it|\log^{2/3}|2 + it|.
\]
This inequality is verified numerically for $|t| < 3$ and is a consequence of 
the bound $|\zeta(1 + it)| \le 62.6 \log^{2/3}t$ for $|t| \ge 3$, 
given in \cite{trudgian_new_2014}.  
Also, on the $2$-line, plainly,
\[
|f(2 + it)| \le |1 + it|\, \zeta(2) < 62.6|3 + it|\log^{2/3}|3 + it|,
\]
for all real $t$. 
By \cite[Lemma 3]{trudgian_improved_2014}, we thus obtain\footnote{We apply \cite[Lemma 3]{trudgian_improved_2014} 
with the following parameter values (in the notation in the cited paper): 
$a = 1$, $b = 2$, $Q = 1$, $\alpha_1 = \beta_1 = 1$, $\alpha_2 = \beta_2 = 2/3$ and $A = B = 62.6$. 
Note that the $A$ and $B$ in \cite[Lemma 3]{trudgian_improved_2014} are different from the $A$ and $B$ in the statement of
our lemma.}
\begin{equation}\label{pl bound}
    |f(s)| \le 62.6|1 + s|\log^{2/3}|1 + s|,\qquad (1 \le \sigma \le 2). 
\end{equation}
Moreover, for  $1 \le \sigma \le 2$ we have the easy estimates,
\begin{equation*}
    \begin{split}
        &\frac{|1 + s|}{|s - 1|} \le \frac{\sqrt{9 + t^2}}{t} \le \sqrt{1 + 9t_0^{-2}},\\
        &\log|1 + s| \le \log\sqrt{9 + t^2} \le \log t + \log\sqrt{1 + 9t_0^{-2}}.
    \end{split}
\end{equation*}
So that, for $1 \le \sigma \le 2$ and $t \ge t_0$, 
 the inequality \eqref{pl bound} implies that
\begin{align*}
|\zeta(\sigma + it)| &\le 62.6\sqrt{1 + 9t_0^{-2}}\left(1 + \frac{\log \sqrt{1 + 9t_0^{-2}}}{\log t_0}\right)\log^{2/3}t
< A \log^Bt,
\end{align*}
where $A$ and $B$ are defined as in \eqref{zeta_AB}, as desired.
\end{proof}
\begin{remark}
Following Titchmarsh's argument \cite[Theorem~5.17]{titchmarsh_theory_1986}, we could make use of \cite[Theorem~1.1]{yang_explicit_2023}. For simplicity, however, we used the Richert-type bound \eqref{ford bound} due to Ford \cite{ford_vinogradov_2002} instead. 
\end{remark}

\begin{lemma}\label{zeta_bound_after_1}
Let $t_0 \ge e^e$ and $d_1 > 0$ be constants. 
Suppose $t\ge t_0$. We have 
\[
\left|\zeta\left(1 + d_1\frac{\log\log t}{\log t} + it\right)\right| \le C'(d_1,t_0)\frac{\log t}{\log\log t},
\]
where $\gamma$ is Euler's constant, and
\begin{equation}\label{C'}
C'(d_1, t_0) = \frac{1}{d_1}\exp\left(\gamma d_1\frac{\log\log t_0}{\log t_0}\right).
\end{equation}
\end{lemma}
\begin{proof}
We use the following uniform bound due to Ramar\'{e} \cite[Lemma~5.4]{ramare2016explicit}:
\[
|\zeta(\sigma + it)| \le \zeta(\sigma) \le \frac{e^{\gamma(\sigma - 1)}}{\sigma - 1},\qquad (\sigma > 1). 
\]
Substituting $\sigma = 1 + d_1\log\log t / \log t$ therefore gives for $t\ge t_0$,
\[
\zeta(\sigma) \le \frac{e^{\gamma d_1 \log\log t / \log t}}{d_1 \log \log t / \log t} \le \frac{1}{d_1}\exp\left(\frac{\gamma d_1\log\log t_0}{\log t_0}\right)\frac{\log t}{\log \log t},
\]
as claimed. 
In the last inequality, we used that $\log\log t/\log t$ reaches its maximum at $t=e^e$ and is monotonically decreasing after that.
\end{proof}

\subsection{Bounding $|\zeta'(1+it)/\zeta(1+it)|$}

We follow the argument of \cite{trudgian2015explicit}, which gives an explicit version of results in \cite[\S3]{titchmarsh_theory_1986}. The method relies on Theorem \ref{borelcaratheodory}.

We will therefore construct concentric disks, centred just to the right of the line $s = 1+it$, and extend their radius slightly to the left and into the critical strip. Ultimately, we want to apply Lemmas \ref{log_deriv_zeta_lem1} and \ref{log_deriv_zeta_lem2} with $f(s)=\zeta(s)$, which will give us the desired results. 

The process is as follows. 
Let $t_0\ge e^e$ and suppose $t'\ge t_0$.
Let $d$ be a real positive constant, which will be chosen later. 
Denote the center of the concentric disks to be constructed by $s' = \sigma' + it'$, 
 where
\begin{equation}\label{sigma'0}
\sigma'= 1+\delta_{t'}, \qquad\text{with}\qquad \delta_{t'} := \frac{d \log\log t' }{ \log t'}.
\end{equation}
Notice that 
$\delta_{t'}$ is decreasing with $t'$ 
and is at most $d/e$.
Let $r>0$ denote the radius of our largest concentric disk. 
In the sequel, $r$ will be determined as function of $s'$. 

In addition, suppose $\sigma' +r \le 1+\epsilon$ 
where $\epsilon\le 1$ is another positive constant to be chosen later.
This implies, for instance, that $r\le \epsilon \le 1$.
In order to enforce this inequality on $\sigma'$ and $r$, it is enough to demand that
\begin{equation}\label{r_eps_cond}
    \frac{d}{e}+r\le \epsilon.
\end{equation}

Let $s=\sigma+it$ be a complex number. 
Aiming to apply Lemma~\ref{log_deriv_zeta_lem1} in the disk $|s-s'|\le r$,
we first seek a valid $A_1$ in that disk. 
Similarly, for Lemma~\ref{log_deriv_zeta_lem2}, we seek
 a valid $A_2$ at the disk center $s=s'$. In applying 
 Lemma~\ref{log_deriv_zeta_lem2},  we moreover need to ensure 
 that the non-vanishing condition on $f(s)=\zeta(s)$ 
 is fulfilled, which we do by way of a zero-free region of zeta.

Let us first determine a valid $A_1$, with the aid of Lemma~\ref{zeta_log_bound}. 
To this end,
let $C>0$ be a parameter restricted by the inequality in \eqref{C assumption}, 
so that (in the notation of Lemma~\ref{zeta_log_bound}) 
$\sigma_t\ge 1/2$ for $t\ge e$.  
Note that $\sigma_t$ is monotonically decreasing for $e\le t\le t_e$
and monotonically increasing thereafter.
So, the maximum of $\sigma_t$ over the interval 
$t'-r\le t\le t'+r$ occurs at one of the end-points. 
For example, if $t'-r\ge t_e$, then the maximum occurs at $t=t'+r$.
And if $t'+r\le t_e$, then the maximum occurs at $t=t'-r$.
With this in mind, we calculate
\begin{equation}\label{C_1_cond0}  
\max_{|t-t'|\le r} \sigma_t= \max\{ \sigma_{t'-r}, \sigma_{t'+r}\} < 1-\frac{C_0\,(\log\log t')^2}{\log t'},
\end{equation}
where, since $r\le \epsilon$,
\begin{equation*}
C_0 := C\, \min \left\lbrace \left(\frac{\log\log (t'-\epsilon)}{\log\log t'}\right)^2, 
\frac{\log t'}{\log (t'+\epsilon)}\right\rbrace.
\end{equation*}
Clearly, $C_0=C_0(t')$ increases monotonically with $t'$ (towards $C$). 
We therefore set
\begin{equation}\label{C_10}
C_1:= C_0(t_0) \qquad\text{and}\qquad 
\sigma_{1,t'} :=  1-\frac{C_1\,(\log\log t')^2}{\log t'}.
\end{equation}
So that, by \eqref{C_1_cond0}, 
no matter $t'$, $r$ and $\epsilon$, so long as they are subject to our conditions, 
$$\sigma_{1,t'}\ge\max_{|t-t'|\le r} \sigma_t.$$
Consequently, Lemma~\ref{zeta_log_bound} gives
that for each $t\in [t'-r,t'+r]$
and any $\sigma \in [\sigma_{1,t'},2]$, 
and on recalling $r\le \epsilon$,
\begin{equation}\label{A3_lb0}
|\zeta(s)| \le A\log^B (t'+r) \leq A_3\log^B t', \qquad A_3 = A\left(1+\frac{\log(1+\epsilon/t')}{\log t'}\right)^B.
\end{equation}

We would like the inequality \eqref{A3_lb0} to hold throughout the disk $|s-s'|\le r$. 
This will be certainly the case if this disk lies entirely in the rectangle specified by $\sigma \in [\sigma_{1,t'},2]$ and $t\in [t'-r,t'+r]$. This follows, in turn, on requiring $\sigma'-r \ge \sigma_{1,t'}$ or, equivalently, that
\begin{equation*}
r \le \frac{C_1(\log\log t')^2}{\log t'} + \delta_{t'}.
\end{equation*}
Therefore, subject to the constraint \eqref{r_eps_cond}, the following choice of $r$ works:
\begin{equation}\label{r_0.5_cond0}
r=r_{t'}:= \left(C_1+\frac{d}{\log\log t'}\right)\frac{(\log\log t')^2}{\log t'}.
\end{equation}
To ensure the constraint \eqref{r_eps_cond}
is met for all $t'\ge t_0$, it suffices that 
\begin{equation}\label{r_eps_cond0}
 \frac{d}{e}+\left(C_1+d\right)\frac{4}{e^2}\le \epsilon \le 1.  
\end{equation}
Here, we used that $(\log \log t')^2/\log t'$ reaches its maximum of $4/e^2$ 
at $t'=t_e$.

Overall, combining \eqref{A3_lb0} and \eqref{r_0.5_cond0} 
with the trivial bound $|\zeta(s')| \ge \zeta(2\sigma')/\zeta(\sigma')$, 
 easily seen on considering the Euler product of zeta, we obtain
that throughout the disk $|s-s'|\le r$,
\begin{align*}
\left| \frac{\zeta(s)}{\zeta(s')}\right| \le &\, A_3\,\frac{\zeta(\sigma')}{\zeta(2\sigma')}\log^B t' =A_3\,\frac{X(\sigma')}{\sigma' -1}\log^B t',
\end{align*}
where 
$$X(\sigma')=\frac{\zeta(\sigma')(\sigma'-1)}{\zeta(2\sigma')}.$$ 
Since $\sigma'=\sigma'(t')$ decreases to $1$ with $t'\ge t_0$, 
$X(\sigma')$ decreases to $1/\zeta(2)$ with $t'$.
Thus, recalling from \eqref{sigma'0} that $\sigma'=1+\delta_{t'}$ 
and observing $A_3$ is also decreasing in $t'\ge t_0$, we obtain 
that throughout the disk $|s-s'|\le r$,
\begin{equation}\label{AX max}
    \left| \frac{\zeta(s)}{\zeta(s')}\right| \le
    \frac{A_{\max}\, X_{\max}}{d}\frac{\log^{B+1} t'}{\log\log t'},
\end{equation}
where
$A_{\max} = A_3(t_0)$ and $X_{\max}=X(\sigma'(t_0))$.

In summary, when applying Lemma \ref{log_deriv_zeta_lem1} in the disk $|s-s'|\le r$, with $r$ as in \eqref{r_0.5_cond0} and subject to \eqref{r_eps_cond0}, we may take
\begin{equation*}
\log A_1 = (B+1)\log\log t' -\log\log\log t' + \log\left( \frac{A_{\max}\,X_{\max}}{d}\right).
\end{equation*}

Next, in preparation for applying 
Lemma \ref{log_deriv_zeta_lem2}, we want to ensure 
that the intersection of the disk $|s-s'|\le r$
 and the right half-plane $\sigma \ge \sigma'-\alpha r$
lies entirely in a zero-free region for zeta. 
By \cite[Corollary~1.2]{yang_explicit_2023}, we have the following zero-free region, 
valid for $t\ge 3$.
\begin{equation*}
\zeta(\sigma+ it)\neq 0\quad\text{for}\quad \sigma> 1-\frac{c_0\log\log t}{\log t}, \quad \text{where}\quad 
c_0=\frac{1}{21.233}.
\end{equation*}
Since $1- c_0 \log\log t/\log t$ is monotonically increasing 
for $t\ge t_0$ and since $t'\ge t_0$ and $r\le \epsilon$, it suffices for our purposes to require
that
\begin{equation*}
\sigma' -\alpha r \ge 1- \frac{c_0\log\log (t'+\epsilon)}{\log (t'+\epsilon)}.
\end{equation*}
In other words, it suffices that
\begin{equation*}
    \alpha \le \frac{1}{r}\left(\delta_{t'} +  \frac{c_0\log\log (t'+\epsilon)}{\log (t'+\epsilon)}\right).
\end{equation*}
Therefore, subject to $\alpha <1/2$ as demanded by Lemma~\ref{log_deriv_zeta_lem2}, 
the following choice of $\alpha$ works.
\begin{equation}\label{alphar}
 \alpha = \frac{1}{r}\cdot \frac{(d+c_1)\log\log t'}{\log t'} = 
 \frac{d+c_1}{d+C_1 \log\log t'},
\end{equation}
where (using monotonicity to deduce the inequality below)
\begin{equation}\label{c0 bound}
c_1 := \frac{c_0 \log t_0}{\log(t_0+\epsilon)},\qquad \text{so that}\qquad 
 c_1 \le c_0.
\end{equation}
In particular, the constraint $\alpha<1/2$ is fulfilled if
\begin{equation}\label{alpha cond}
    \frac{d+c_1}{d+C_1 \log\log t_0}<\frac{1}{2}.
\end{equation}

Now, to determine $A_2$, we utilise \cite{delange1987remarque}, 
which asserts that if $\sigma >1$, then
\begin{equation*}
\left| \frac{\zeta'(s)}{\zeta(s)}\right| \le -\frac{\zeta'(\sigma)}{\zeta(\sigma)} < \frac{1}{\sigma -1}.
\end{equation*}
Taking $s=s'$ in this inequality leads to
\begin{equation*}
\left| \frac{\zeta'(s')}{\zeta(s')}\right| < \frac{1}{\delta_{t'}} =\frac{\log t'}{d \log\log t'}.
\end{equation*}
So, we may take $A_2$ in Lemma~\ref{log_deriv_zeta_lem2} to be
$$A_2=\frac{\log t'}{d \log\log t'}.$$

Lastly, in applying Lemma \ref{log_deriv_zeta_lem2},
we need to be able to reach the $1$-line from our position at $s'$. 
So, we set 
$$\beta = \frac{\delta_{t'}}{\alpha r} =\frac{d}{c_1+d}.$$ 
Clearly, $\beta<1$, as required by Lemma~\ref{log_deriv_zeta_lem2}.
The conclusion of the lemma therefore holds throughout the disk 
$|s-s'|\le \alpha\beta r = \delta_{t'}$. This includes, in particular,
the horizontal line segment $s=\sigma + it'$, with $1\le \sigma\le 1+2\delta_{t'}$.

Putting it all together, and making the substitution $t'\to t$, 
Lemma \ref{log_deriv_zeta_lem2} hence furnishes the bound
\begin{equation}\label{lemma6 bound}
\left| \frac{\zeta'(s)}{\zeta(s)}\right| \le \frac{8\beta}{(1-\beta)(1-2\alpha)^2}
 \, \frac{\log A_1}{r}+ \frac{1+\beta}{1-\beta}\,A_2,
\end{equation}
valid in the region
\begin{equation*}
1 \le \sigma \le 1+ \frac{2d \log\log t }{\log t},\qquad t\ge t_0.
\end{equation*}

We now input the values we determined  
for $A_1$, $A_2$, $r$, $\alpha$ and $\beta$ into \eqref{lemma6 bound}.  We also note
\begin{equation*}
    \frac{1+\beta}{1-\beta}=1+\frac{2d}{c_1},\qquad \text{so that}\qquad
    \frac{1+\beta}{1-\beta}\,A_2 =\left( \frac{1}{d}+\frac{2}{c_1}\right)\frac{\log t}{\log\log t},
\end{equation*}
as well as
\begin{equation*}
\frac{8\beta}{1-\beta}=\frac{8d}{c_1},\qquad 
    \frac{1}{(1-2\alpha)^2} =\left(\frac{C_1\log\log t+d}{C_1\log \log t-d-2c_1} \right)^2.
\end{equation*}
Hence, on observing that $1/(1-2\alpha)^2$
is decreasing in $t$, we obtain  the following lemma.

\begin{lemma}\label{zeta'/zeta_main_lem}
Let $t_0\ge e^e$ and suppose that $t\ge t_0$.
Let $C$, $d$ and $\epsilon$ be any positive constants satisfying the constraints 
\eqref{C assumption}, \eqref{r_eps_cond0} and \eqref{alpha cond}. If
$$1 \le \sigma \le 1+ \frac{2d \log\log t}{\log t},$$
then
\begin{equation*}
\left| \frac{\zeta'(\sigma +it)}{\zeta(\sigma+it)}\right| \le Q_1(C,d,\epsilon,t_0)\frac{\log t}{\log\log t},
\end{equation*}
where 
\begin{align*}
Q_1(C,d,\epsilon,t_0) := \lambda_1 \left( B+1 
+ \frac{1}{\log\log t_0}\log\Bigg( \frac{A_{\max}X_{\max}}{d}\Bigg)\right)
+ \lambda_2.
\end{align*}
Here,
\begin{equation*}
\begin{split}
    \lambda_1 := \frac{1}{C_1}\cdot\frac{8d}{c_1} \cdot
    \left(\frac{C_1\log\log t_0+d}{C_1\log \log t_0-d-2c_1} \right)^2, \qquad
    \lambda_2 :=\frac{1}{d}+\frac{2}{c_1}.
\end{split}
\end{equation*}
The number $B$ is defined as in Lemma~\ref{zeta_log_bound} and depends on $C$.
The numbers $A_{\max}$ and $X_{\max}$ are defined as in \eqref{AX max} and depend on $\epsilon$ and $t_0$.
The number $C_1$ is defined in \eqref{C_10}, and depends on $C$, $\epsilon$ and $t_0$. 
The number $c_1$ is defined in \eqref{c0 bound} and 
depends on $\epsilon$, $t_0$ and $c_0=1/21.233$. 
\end{lemma}

\subsection{Bounding $|1/\zeta(1+it)|$}

Moving from a bound on the logarithmic derivative of $\zeta$ to one on the reciprocal of $\zeta$ is done in the usual way (see \cite[Theorem~3.11]{titchmarsh_theory_1986}), with some improvements using the trigonometric polynomial (see \cite[Proposition~A.2]{carneiro2022optimality}).

\begin{lemma}\label{1/zeta_main_lem}
Let $t_0\ge e^e$ and $d>0$ be constants.
Suppose that for each $t\ge t_0$,
\begin{equation*}
\text{if}\quad 1 \le \sigma \le 1+\frac{2d\log\log t}{\log t},\quad\text{then}\quad 
\left| \frac{\zeta'(\sigma +it)}{\zeta(\sigma+it)}\right| \le R_1\frac{\log t}{\log\log t}.
\end{equation*}
Then for any $t\ge t_0$ and any real number $d_1$ such that $0< d_1 \le 2d$, we have
\begin{equation*}
\left| \frac{1}{\zeta(1+it)}\right| \le Q_2(d_1,t_0)\frac{\log t}{\log\log t},
\end{equation*}
where
\begin{equation*}
Q_2(d_1,t_0) = \exp(R_1 d_1)\left(\frac{\log\log t_0}{\log t_0} +\frac{1}{d_1}\right)^{3/4}\left(C'(d_1, 2t_0)\left(\frac{\log 2}{\log t_0} +1\right)\right)^{1/4}.
\end{equation*}
The number $C'$ is defined as in \eqref{C'}.
\end{lemma}
\begin{proof}
Suppose $t\ge t_0$.
Let $d_1$ be a real positive parameter such that $d_1\le 2d$.  
Let $\delta_1 = d_1\log\log t/ \log t$. 
It is easy to see that
\begin{align*}
\log\left| \frac{1}{\zeta(1+it)}\right| 
&= -\textup{Re} \log \zeta(1+it) \\
&= - \textup{Re} \log \zeta\left( 1+ \delta_1 +it\right) + \int_1^{1+\delta_1} \textup{Re}\left( \frac{\zeta'}{\zeta}(\sigma+it)\right) \text{d}\sigma. \nonumber 
\end{align*}
So, using our suppositions, we have
\begin{equation}\label{1/zeta_int}
    \log\left| \frac{1}{\zeta(1+it)}\right| 
    \le -\log\left| \zeta\left( 1+\delta_1 +it \right)\right| + R_1 d_1.
\end{equation}

On the other hand, for $\sigma>1$, and using the classical nonnegativity argument involving the trigonometric polynomial $3 + 4\cos\theta + \cos2\theta$ \cite[Section 3.3]{titchmarsh_theory_1986}, gives that 
$$\zeta^3(\sigma)|\zeta^4(\sigma+it)\zeta(\sigma +2it)| \ge 1.$$ 
Combining this with the inequality $\zeta(\sigma) \le \sigma/(\sigma-1)$ and Lemma \ref{zeta_bound_after_1},
we therefore obtain, on taking $\sigma = 1+\delta_1$,
\begin{equation}\label{1/zeta delta 1}
\begin{split}
    \left| \frac{1}{\zeta(1 + \delta_1 +it)}\right| &\le |\zeta(1 + \delta_1)|^{3/4}|\zeta(1 + \delta_1 + 2it)|^{1/4}\\
&\le \left( \frac{1 + \delta_1}{\delta_1}\right)^{3/4}\left(C'(d_1, 2t_0)\frac{\log 2t}{\log\log 2t}\right)^{1/4}.
\end{split}
\end{equation}
We now observe
\begin{equation}\label{1/zeta delta 2}
\begin{split}
    \frac{1 + \delta_1}{\delta_1} &\le \left(\frac{\log\log t_0}{\log t_0}+\frac{1}{d_1}\right) \frac{\log t}{\log\log t},\\
    \frac{\log 2t}{\log\log 2t} &\le \left(\frac{\log 2} {\log t_0} +1\right) \frac{\log t}{\log\log t}.
\end{split}
\end{equation}
So that, combining \eqref{1/zeta delta 2} and \eqref{1/zeta delta 1}, and noting 
by \eqref{1/zeta_int} we have
\begin{align*}
\left| \frac{1}{\zeta(1+it)}\right| &\le \frac{\exp(R_1 d_1)}{|\zeta(1+\delta_1 +it)|}, 
\end{align*}
the lemma follows.
\end{proof}

\subsection{Numeric calculations}

We define the following upper bounds:
\begin{equation*}
Q_1(C,d,\epsilon,t_0)\le R_1    \qquad \text{ and }\qquad Q_2(d_1,t_0)\le R_2.
\end{equation*}

The numeric calculations for $R_1$ and $R_2$ proceed as follows. First fix $t_0\ge e^e$ and a positive $\epsilon \le 1$. For the purposes of obtaining an upper bound $R_1$, we then optimise $Q_1(C,d,\epsilon,t_0)$ in Lemma \ref{zeta'/zeta_main_lem} over $C$ and $d$, all the while ensuring $C$, $d$, and $\epsilon$ are subject to our constraints \eqref{C assumption}, \eqref{r_eps_cond0} and \eqref{alpha cond}

Thereafter, with these values of $d$ and $R_1$, the relationship $d_1\le 2d$ allows us to obtain $R_2$ by optimising $Q_2(d_1,t_0)$ in Lemma \ref{1/zeta_main_lem} over $d_1$.

Our choices of
\begin{align*}
t_0 = 500, \qquad \epsilon = 0.517, 
\end{align*}
\begin{align*}
C = \frac{e^2}{8}, \qquad d = 0.018, \qquad d_1 = 0.0065,
\end{align*}
give us 
\begin{alignat*}{3}
r &\le 0.502,\qquad &&\alpha \le 0.0381,\qquad &&\beta = 0.274\ldots,
\end{alignat*}
and the upper bounds
\begin{equation*}
R_1 = 154.5, \qquad R_2 = 431.7.
\end{equation*}

In summary, we have for $t\ge 500$,
\begin{equation}\label{1/zeta_t0_500}
\frac{1}{|\zeta(1 + it)|} \le 431.7 \frac{\log t}{\log \log t}.
\end{equation}
Furthermore, in the region
\begin{equation*}
1 \le \sigma \le 1 + \frac{9}{250}\frac{\log\log t}{\log t},
\end{equation*}
we have for $t\ge 500$,
\begin{equation}\label{zeta'/zeta_t0_500}
\left|\frac{\zeta'}{\zeta}(\sigma+it)\right| \le 154.5 \frac{\log t}{\log \log t}.
\end{equation}

Finally, to cover the range from $3\le t\le 500$ in \eqref{1/zeta_t0_500}, we refer to the proof of Proposition $A.2$ in \cite{carneiro2022optimality}, where via interval arithmetic, they computed that for $2\le t \le 500$,
\begin{equation*}
\frac{1}{|\zeta(1 + it)|} \le 2.079 \log t.
\end{equation*}
Comparing this with the estimate in \eqref{1/zeta_t0_500}, we see after a simple verification that the bound
\begin{equation*}
\frac{1}{|\zeta(1 + it)|} \le 431.7 \frac{\log t}{\log \log t}
\end{equation*}
holds for all $t\ge 3$.

This concludes our proof of Theorem \ref{main theorem 1}. These bounds 
can certainly be improved. Some possible avenues of attack would be to treat the sum over zeros in the proof of Lemma \ref{log_deriv_zeta_lem2} more carefully, or to use a larger zero-free region.

\begin{remark}
It is possible with our methods to directly get a bound like \eqref{zeta'/zeta_t0_500} for $t<500$, especially if one were not concerned with the size of $R_2$. For instance, choosing
$t_0 = 16$, $\epsilon = 0.508$, $C = e^2/8$, $d = 0.0147$ and $d_1 = 0.0056$,
would return for $t\ge 16$ that if 
\begin{equation*}
1 \le \sigma \le 1 + \frac{147}{5000}\frac{\log\log t}{\log t},
\end{equation*}
then
\begin{equation*}
\left|\frac{\zeta'}{\zeta}(\sigma+it)\right| \le 177.5 \frac{\log t}{\log \log t}.
\end{equation*}
While this would cause a modest jump in $R_1$, the size of 
$R_2$ would increase significantly to $R_2= 511.1$.
This is because $R_2$ increases exponentially with $R_1$, as evident from Lemma \ref{1/zeta_main_lem}. The dominant factor there is $C'$, and any improvement in that quantity would help greatly. 
\end{remark}

\section{Background results}

\subsection{Background results related to Sections \ref{Theorem 1 proof}-\ref{Proposition proof}}

\begin{lemma}\cite[Lemma 2.1]{francis2021investigation}\label{harmonic sum}
    Let $\gamma$ denote Euler's constant. For $N\ge 1$, 
    \begin{equation*}
        \sum_{n = 1}^N \frac{1}{n} < \log N + \gamma +\frac{1}{2N}-\frac{1}{12N^2}+\frac{1}{64N^4}.
    \end{equation*}
\end{lemma}

\begin{theorem}\cite[Theorem 3.9]{patel_explicit_2022}\label{patel theorem}
    For $t>0$ and $n_1=\lfloor \sqrt{t/(2\pi)}\rfloor$ we have
    $$|\zeta(1+it)|\le \left|\sum_{n=1}^{n_1} \frac{1}{n^{1+it}}\right|+\frac{g(t)}{t^{1/2}}\left|\sum_{n=1}^{n_1} \frac{1}{n^{-it}}\right|+\mathcal{R}$$
    where 
    $$g(t)=\sqrt{2\pi}\exp\left(\frac{5}{3t^2}+\frac{\pi}{6t}\right),$$
    and 
\begin{equation*}
    \begin{split}
         \mathcal{R}=\mathcal{R}(t):=&\frac{1}{t^{1/2}}\left(\sqrt{\frac{\pi}{2}}+\frac{g(t)}{2}\right)
    +\frac{1}{t}\left(9\sqrt{\frac{\pi}{2}}+\frac{g(t)}{\sqrt{\pi(3-2\log 2)}}\right)\\
    &+\frac{1}{t^{3/2}}\left(\frac{968\pi^{3/2}+g(t)242\pi}{700}\right).
    \end{split}
\end{equation*}
\end{theorem}

Next we review some explicit estimates of exponential sums from the literature. The following lemma is a specialised third-derivative test for a phase function commonly encountered when estimating the Riemann zeta-function. It is due to \cite[Lemma 2.6]{hiary_improved_2022} which builds on \cite{hiary_explicit_2016}.

\begin{lemma}[Explicit third derivative test]\cite[Lemma 2.6]{hiary_improved_2022}\label{corput lemma}
Let $r$ and $K$ be positive integers, and let $t$ and $K_0$ be positive numbers.
    Suppose $K \ge K_0 > 1$. Let 
\begin{align*}
f(x) := \frac{t}{2\pi}\log\left(rK + x\right), \qquad W :=
    \frac{\pi(r + 1)^3K^3}{t},\qquad\lambda := \frac{(1 + r)^3}{r^3},
\end{align*}
so that $1/W \le |f'''(x)| \le \lambda/W$ for $0 \le x \le K$. Then, for each positive integer $L \le K$
and any $\eta > 0$,
\[
\left|\sum_{n = 0}^{L - 1}e^{2\pi i f(n)}\right|^2 \le \left(\frac{K}{W^{1/3}} + \eta\right)(\alpha K + \beta W^{2/3})
\]
where
\begin{align*}
\alpha &= \alpha(W, \lambda) := \frac{1}{\eta} + \frac{\eta\mu}{3W^{1/3}} + \frac{\mu}{3W^{2/3}} + \frac{32\mu}{15\sqrt{\pi}}\sqrt{\eta + W^{-1/3}},\\
\beta &= \beta(W) := \frac{32}{3\sqrt{\pi \eta}} + \left(2 - \frac{4}{\pi}\right)\frac{1}{W^{1/3}},\\
\mu &= \mu(\lambda) := \frac{1}{2}\lambda^{2/3}\left(1 + \frac{1}{(1 - K_0^{-1})\lambda^{1/3}}\right).
\end{align*}
\end{lemma}

\begin{remark}
    If $W\le 1$, say, then $\alpha$ will be greater than $1$, and so $\alpha L^2 \ge L^2$. Hence, for small $W$, it is better to use the triangle inequality than Lemma~\ref{corput lemma}.
\end{remark}

\begin{remark}
  We mention that there is a typo in the statement of \cite[Lemma 2.6]{hiary_improved_2022} where $N$ in the statement of the lemma should be $0$. 
This lemma is only applied with $N=0$ in \cite{hiary_improved_2022} in any case.\footnote{In passing, let us also point out that on \cite[211]{hiary_improved_2022}, the minimum in the definition of $y_j$ should be against $L-m$ rather than $L-m-1$. This does not affect Equation~5.4 on p.~211 since, as stated in \cite[Lemma 2.5]{hiary_improved_2022}, the bound on $|g''(x)|$ holds up to $K - m$, not only $K-m-1$. None of the results in \cite{hiary_improved_2022} are affected.} 
\end{remark}

The next two lemmas respectively represent another third derivative test and a $k$th derivative test for $k \ge 4$, due to \cite{yang_explicit_2023}. The reason we state this additional third derivative test is because $A_3$ and $B_3$ 
    comprise the boundary cases for the recursive formulas in Lemma~\ref{kth_deriv_test}.

\begin{lemma}[Another explicit third derivative test] \cite[Lemma 2.4]{yang_explicit_2023}\label{third_deriv_test}
Let $a$ and $N$ be integers such that $N\ge 1$. 
Suppose $f(x)$ has three continuous derivatives and $f'''(x)$ is monotonic over the interval 
$(a, a + N]$. Suppose further that  
there are numbers $\lambda_3 > 0$ and $\omega > 1$ such that $\lambda_3 \le |f'''(x)| \le \omega \lambda_3$ for all $x \in (a, a + N]$. Then for any $\eta > 0$,  
\[
\left|\sum_{n=a+1}^{a+N} e^{2\pi i f(n)}\right| \le A_3\omega^{1/2}N\lambda_3^{1/6} + B_3 N^{1/2}\lambda_3^{-1/6}
\]
where 
\begin{equation*}
\begin{split}
    &A_3=A_3(\eta, \omega) := \sqrt{\frac{1}{\eta \omega} + \frac{32}{15\sqrt{\pi}}\sqrt{\eta + \lambda_0^{1/3}} + \frac{1}{3}\left(\eta + \lambda_0^{1/3}\right)\lambda_0^{1/3}}\delta_3,
    \end{split}
\end{equation*}
\begin{equation*}
\begin{split}
    &B_3=B_3(\eta) := \frac{\sqrt{32}}{\sqrt{3}\pi^{1/4}\eta^{1/4}}\delta_3,
\end{split}
\end{equation*}
and $\lambda_0$ and $\delta_3$ are defined by
\begin{equation*}
\begin{split}
    &\lambda_0=\lambda_0(\eta,\omega) := \left(\frac{1}{\eta} + \frac{32\eta^{1/2}\omega}{15\sqrt{\pi}}\right)^{-3},
\end{split}
\end{equation*}
\begin{equation*}
\begin{split}
    &\delta_3 =\delta_3(\eta) := \sqrt{\frac{1}{2} + \frac{1}{2}\sqrt{1 + \frac{3}{8}\pi^{1/2}\eta^{3/2}}}.
\end{split}
\end{equation*}
\end{lemma}

\begin{lemma}[General explicit $k$-th derivative test] \cite[Lemma 2.5]{yang_explicit_2023}\label{kth_deriv_test} Let $a$, $N$ and $k$ be integers 
such that $N \ge 1$ and $k\ge 4$. Suppose $f(x)$ is equipped with $k$ continuous derivatives
and $f^{(k)}(x)$is  monotonic over the interval $(a,a+N]$.
Suppose further that there are numbers $\lambda_k > 0$ and $\omega > 1$ such that
$\lambda_k \le |f^{(k)}(x)| \le \omega \lambda_k$ for all $x \in (a, a+N]$. Then for any $\eta>0$,
\[
\left|\sum_{n=a+1}^{a+N} e^{2\pi i f(n)}\right| \le A_k \omega^{2/K}N\lambda_k^{1/(2K - 2)} + B_kN^{1 - 2/K}\lambda_k^{-1/(2K - 2)},
\]
where $K = 2^{k - 1}$, $A_k$ and $B_k$ are defined recursively via the formulas
\begin{equation}\label{s1_Ak_recursive_defn}
A_{k + 1}(\eta, \omega) := \delta_k\left(\omega^{-1/K} + \frac{2^{19/12}(K - 1)}{\sqrt{(2K - 1)(4K - 3)}}A_{k}(\eta, \omega)^{1/2}\right),
\end{equation}
\begin{equation}
B_{k + 1}(\eta) := \delta_k\frac{2^{3/2}(K - 1)}{\sqrt{(2K - 3)(4K - 5)}} B_{k}(\eta)^{1/2},
\end{equation}
with $A_3$ and $B_3$ as in Lemma \ref{third_deriv_test}, and $\delta_k$ defined by
\begin{equation}
\delta_k=\delta_k(\eta) := \sqrt{1 + \frac{2}{2337^{1 - 2/K}}\left(\frac{9\pi}{1024}\eta\right)^{1/K}}.
\end{equation}
\end{lemma}

\begin{lemma}\label{CD_deriv_test}
Let $a$, $b$ and $k$ be integers 
such that $a>0$, $b>a$ and $k\ge 3$. Let $t>0$. 
Let $h > 1$ be a number such that $b\le ha$.
Then for any $\eta>0$, 
$$
\left|\sum_{a < n \le b}\frac{1}{n^{1 +it}}\right| \le C_k a^{- k/(2K - 2)}t^{1/(2K - 2)} 
+ D_k a^{k/(2K - 2)- 2/K}t^{-1/(2K - 2)},
$$
where 
\begin{equation*}
\begin{split}
C_k=C_k(\eta, h) &:= A_k(\eta, h^k) h^{2k/K - k / (2K - 2)}(h - 1)\left(\frac{(k - 1)!}{2\pi}\right)^{\frac{1}{2K - 2}},\\
D_k=D_k(\eta, h) &:= B_k(\eta) h^{k/(2K - 2)}(h - 1)^{1 - 2/K}\left(\frac{2\pi}{(k - 1)!}\right)^{\frac{1}{2K - 2}}.
\end{split}
\end{equation*}
and $A_k(\eta, h^k)$ and $B_k(\eta)$ defined in accordance with 
Lemma~\ref{kth_deriv_test}.
\end{lemma}
\begin{proof}
This is a special case of Lemma~\ref{kth_deriv_test} and Lemma~\ref{third_deriv_test} 
with the phase function
$$f(x) = -\frac{t}{2\pi}\log x,\qquad \text{and hence}\qquad f^{(k)}(x)=(-1)^k \frac{t}{2\pi}\frac{(k-1)!}{x^k}.$$
Since $a < b \le ha$ by assumption,
$\lambda_k \le |f^{(k)}(x)| \le \omega \lambda_k$ for 
$a < x \le b$, where
\[
\lambda_k = \frac{t}{2\pi} \frac{(k - 1)!}{h^k a^k},\qquad \omega=h^k.
\]
Suppose that $k\ge 4$. Then, we may apply Lemma~\ref{kth_deriv_test} with these values of $\lambda_k$ and $\omega$, 
and with $N =b-a$. This furnishes a bound that increases with $N$. 
Noting that $N\le (h - 1)a$, we thus obtain
\begin{align*}
\left|\sum_{a < n \le b}\frac{1}{n^{it}}\right| &\le A_k(\eta,h^k)(h^{k})^{2/K}(h - 1)a
\left(\frac{t}{2\pi}\frac{(k - 1)!}{h^k a^k}\right)^{1/(2K - 2)} \\
&\qquad\qquad + 
B_k(\eta)((h - 1)a)^{1 - 2/K}\left(\frac{t}{2\pi}\frac{(k - 1)!}{h^k a^k}\right)^{-1/(2K - 2)} \\
&= C_k a^{1 - k/(2K - 2)}t^{1/(2K - 2)} 
+ D_k a^{1 - 2/K + k/(2K - 2)}t^{-1/(2K - 2)}.
\end{align*}
The desired result hence follows from partial summation in this case. If $k=3$, then 
the result follows from Lemma~\ref{third_deriv_test} using a similar calculation.
\end{proof}

\subsection{Background results related to Section \ref{Theorem 2 proof}}

Here we give some preliminary lemmas for our proofs of estimates on $|\zeta'(1+it)/\zeta(1+it)|$ and $|1/\zeta(1+it)|$. We reiterate that Section \ref{Theorem 2 proof} reuses variable names from previous sections. These variables mean different things in the confines of this section as this section is independent from Sections {\ref{Theorem 1 proof}, {\ref{A0 A2 calc}}, and {\ref{Proposition proof}}.

The first two lemmas, which form the basis of our method, are built upon a theorem of Borel and Carath\'eodory.
This theorem enables us to deduce an upper bound for the modulus of a function and its derivatives on a circle, from bounds for its real part on a larger concentric circle.

\begin{theorem}[Borel--Carath\'{e}odory]\label{borelcaratheodory}
Let $s_0$ be a complex number. 
Let $R$ be a positive  number, possibly depending on $s_0$. 
Suppose that the function $f(s)$ is analytic in  a region containing
the disk $|s-s_0|\le R$. Let $M$ denote 
the maximum of $\textup{Re}\, f(s)$ on the boundary $|s-s_0|=R$. 
Then, for any $r\in (0,R)$
and any $s$ such that $|s-s_0|\le r$,
\begin{equation*}
|f(s)| \le \frac{2r}{R-r}M + \frac{R+r}{R-r}|f(s_0)|.
\end{equation*}
If in addition $f(s_0)=0$, then 
for any $r\in (0,R)$ and any $s$ such that $|s-s_0|\le r$,
\begin{equation*}
|f'(s)| \le \frac{2R}{(R-r)^2}M.
\end{equation*}
\end{theorem}
\begin{proof}
Proofs for estimating both $|f(s)|$ and $|f'(s)|$ are available in \cite[\S5.5]{titchmarsh1939theory}. However, if we include the restriction $f(s_0)=0$, we can refer to \cite[Lemma~6.2]{montgomery2007multiplicative}, which provides alternative proofs, resulting in a sharper estimate for $|f'(s)|$.
\end{proof}

\begin{lemma}\label{log_deriv_zeta_lem1}
Let $s_0$ be a complex number.
Let $r$ and $\alpha$ be positive numbers (possibly depending on $s_0$) 
such that $\alpha <1/2$.
Suppose that the function $f(s)$ is analytic in  a region containing 
the disk $|s-s_0|\le r$. Suppose further there is a number $A_1$ independent of $s$ such that,
$$\left|\frac{f(s)}{f(s_0)}\right| \le A_1,\qquad |s-s_0|\le r.$$
Then, for any $s$ in the disk $|s-s_0|\le \alpha r$
we have
\begin{equation*}
\left| \frac{f'(s)}{f(s)} - \sum_\rho \frac{1}{s-\rho}\right| \le \frac{4\log A_1}{r(1 -2\alpha)^2},
\end{equation*}
where $\rho$ runs through the zeros of $f(s)$ in the disk
$|s -s_0|\le \tfrac{1}{2} r$, counted with multiplicity.
\end{lemma}
\begin{proof}
We follow the proof in \cite[Section 3]{titchmarsh_theory_1986}.  
Let $g$ be the function 
$$g(s)=f(s)\prod_\rho \frac{1}{s-\rho},$$
where $\rho$ in the (finite) product runs through the zeros  of $f(s)$, counted with multiplicity, 
that satisfy $|\rho -s_0|\le\tfrac{1}{2} r$.
Since by construction the poles and zeros cancel, $g$ is analytic in $|s-s_0|\le r$, and 
does not vanish in the disk $|s-s_0|\le \tfrac{1}{2} r$.
Therefore, the function
$$h(s)=\log \frac{g(s)}{g(s_0)},$$
where the logarithm branch is determined by $h(s_0)=0$, 
is analytic in some region containing the disk $|s-s_0|\le \tfrac{1}{2} r$. 

Now, on the circle $|s-s_0|=r$, we have $|s-\rho|\ge \tfrac{1}{2} r\ge|s_0-\rho|$. 
So, on this circle,
\begin{equation*}
\left| \frac{g(s)}{g(s_0)} \right| = \left| \frac{f(s)}{f(s_0)}\prod_\rho \left( \frac{s_0-\rho}{s-\rho}\right) \right|\le \left| \frac{f(s)}{f(s_0)} \right|\le A_1.
\end{equation*} 
The inequality also holds in the interior of the circle, 
by the maximum modulus principle. Hence,
$\textup{Re}\, h(s)\le \log A_1$
throughout the disk $|s-s_0|\le r$.

Given this, we apply Theorem \ref{borelcaratheodory}
to $h(s)$ with the concentric circles centered at $s_0$ of radii $\alpha r$ and $\frac{1}{2}r$. 
This yields throughout the disk $|s-s_0|\le \alpha r$ that
\begin{equation}\label{h bound}
|h'(s)|\le \frac{4\log A_1}{r(1-2\alpha)^2}.
\end{equation}
On the other hand, since $g(s_0)$ is constant, we have
\begin{equation*}
|h'(s)|=\left| \frac{g'(s)}{g(s)}\right| = \left| \frac{f'(s)}{f(s)} -\sum_\rho \frac{1}{s-\rho}\right|.
\end{equation*}
Combining this formula with the inequality \eqref{h bound} completes the proof.
\end{proof}
\begin{lemma}\label{log_deriv_zeta_lem2}
Let $s$ and $s_0$ be complex numbers with real parts $\sigma$ and $\sigma_0$, respectively.
Let $r$, $\alpha$, $\beta$, $A_1$ and $A_2$ be positive numbers, 
possibly depending on $s_0$, such that $\alpha<1/2$ and $\beta<1$.
Suppose that the function $f(s)$ satisfies 
the conditions of Lemma \ref{log_deriv_zeta_lem1} with $r$, $\alpha$ and $A_1$, 
and that
$$\left| \frac{f'(s_0)}{f(s_0)}\right| \le A_2.$$ 
Suppose, in addition, that $f(s)
\neq 0$ for any $s$ in both 
the disk $|s-s_0|\le r$ and the right half-plane $\sigma \ge \sigma_0 - \alpha r$. 
Then, for any $s$ in the disk $|s-s_0|\le \alpha\beta r$,
\begin{equation*}
\left| \frac{f'(s)}{f(s)}\right| \le \frac{8\beta\log A_1}{r(1-\beta)(1-2\alpha)^2} + \frac{1+\beta}{1-\beta}A_2.
\end{equation*}
\end{lemma}
\begin{proof}
Using the inequality $-|z|\le - \textup{Re}\,z\le |z|$, valid for any complex number $z$,
together with Lemma \ref{log_deriv_zeta_lem1} in $|s-s_0|\le \alpha r$, we have 
throughout this disk,
\begin{align}\label{ff' bound}
-\textup{Re}\, \frac{f'(s)}{f(s)}  &\le \frac{4\log A_1}{r(1-2\alpha)^2}- \sum_\rho\textup{Re}\, \frac{1}{s-\rho},
\end{align}
where $\rho$ runs through the zeros  of $f$ (with multiplicity) 
satisfying $|\rho -s_0|\le\tfrac{1}{2} r$.  
By the assumption on the nonvanishing of $f$,  each $\rho$ 
in this sum satisfies $\textup{Re}\, \rho < \sigma_0 -\alpha r$.
On the other hand, all $s$ in the disk $|s-s_0|\le \alpha r$ clearly satisfy $\sigma\ge \sigma_0-\alpha r$. Hence,
$$\sum_{\rho}\textup{Re}\, \frac{1}{s-\rho} = \sum_{\rho}\frac{\sigma - \textup{Re}\, \rho}{|s-\rho|^2} > 0,$$
throughout the disk $|s-s_0|\le \alpha r$.
We may thus drop the sum in \eqref{ff' bound} and still obtain a valid upper bound, call it $M_1$.

Now, $-f'(s)/f(s)$ is analytic  in the disk $|s-s_0|\le \alpha r$ since by assumption
$f$ does not vanish there.
We can therefore apply Theorem~\ref{borelcaratheodory}
to $-f'(s)/f(s)$ using the two concentric circles $|s-s_0|\le \alpha\beta r$ and $|s-s_0|\le \alpha r$, 
and the upper bound $M_1\ge M$ on the circle 
$|s-s_0|\le \alpha r$, which yields the result.
\end{proof}

\section*{Acknowledgements}
We would like to thank Fatima Majeed for pointing out errors in an early version of this manuscript and giving us some helpful suggestions. 

\printbibliography

\end{document}